\documentclass[11pt,english]{extarticle}
\usepackage[]{fontenc}
\usepackage[latin9]{inputenc}
\usepackage{geometry}
\usepackage{mathrsfs}
\usepackage{amsmath}
\usepackage{amsthm}
\usepackage{amssymb}
\usepackage{stackrel}
\usepackage{esint}

\makeatletter
  \theoremstyle{remark}
  \newtheorem{rem}{\protect\remarkname}
\theoremstyle{plain}
\newtheorem{thm}{\protect\theoremname}
  \theoremstyle{definition}
  \newtheorem{defn}{\protect\definitionname}
  \theoremstyle{plain}
  \newtheorem{lem}{\protect\lemmaname}

\usepackage{dsfont}
\usepackage{relsize}
\usepackage{subdepth}
\allowdisplaybreaks

\usepackage{url}

\DeclareFontFamily{OT1}{pzc}{}
\DeclareFontShape{OT1}{pzc}{m}{it}{<-> s * [1.200] pzcmi7t}{}
\DeclareMathAlphabet{\mathpzc}{OT1}{pzc}{m}{it}

\usepackage[noadjust]{cite}
\usepackage{filecontents}

\usepackage{stfloats}

\usepackage{babel}

\renewcommand\footnotemark{}

\makeatother

  \providecommand{\definitionname}{Definition}
  \providecommand{\lemmaname}{Lemma}
  \providecommand{\remarkname}{Remark}
\providecommand{\theoremname}{Theorem}

\begin{document}

\title{\textbf{Asymptotically Optimal Discrete Time}\\
\textbf{Nonlinear Filters From Stochastically Convergent }\\
\textbf{State Process Approximations}}

\author{Dionysios S. Kalogerias\thanks{The Authors are with the Department of Electrical \& Computer Engineering,
Rutgers, The State University of New Jersey, 94 Brett Rd, Piscataway,
NJ 08854, USA. e-mail: \{d.kalogerias, athinap\}@rutgers.edu.}\thanks{This paper constitutes an extended and most up-to-date version of
\cite{KalPetNonlinear_2015}. The work presented herein is supported
by the National Science Foundation (NSF) under Grant CNS-1239188.} and Athina P. Petropulu}

\date{May 2016}

\maketitle
\textbf{\vspace{-30pt}
}
\begin{abstract}
We consider the problem of approximating optimal in the Minimum Mean
Squared Error (MMSE) sense nonlinear filters in a \textit{discrete
time} setting, exploiting properties of stochastically convergent
state process approximations. More specifically, we consider a class
of nonlinear, partially observable stochastic systems, comprised by
a (possibly nonstationary) hidden stochastic process (the state),
observed through another conditionally Gaussian stochastic process
(the observations). Under general assumptions, we show that, given
an approximating process which, for each time step, is stochastically
convergent to the state process, an approximate filtering operator
can be defined, which converges to the true optimal nonlinear filter
of the state in a strong and well defined sense. In particular, the
convergence is compact in time and uniform in a \textit{completely
characterized} measurable set of probability measure almost unity,
also providing a purely quantitative justification of Egoroff's Theorem
for the problem at hand. The results presented in this paper can form
a common basis for the analysis and characterization of a number of
heuristic approaches for approximating optimal nonlinear filters,
such as approximate grid based techniques, known to perform well in
a variety of applications.
\end{abstract}
\textbf{\textit{$\quad$}}\textbf{Keywords.} Approximate Nonlinear
Filtering, Hidden Models, Partially Observable Systems, Stochastic
Processes, ${\cal C}$-Weak Convergence, Change of Probability Measures.

\section{Introduction}

Nonlinear stochastic filtering refers to problems in which a stochastic
process, usually called the state, is partially observed as a result
of measuring another stochastic process, usually called the observations
or measurements, and the objective is to estimate the state or some
functional of it, based only on past and present observations. The
nonlinearity is due to the general, possibly non Gaussian nature of
the state and observations processes, as well as the fact that, in
general, the state may be partially observed as a nonlinear functional
of the observations. Usually, nonlinear state estimators are designed
so as to optimize some performance criterion. Most commonly, this
corresponds to the Minimum Mean Squared Error (MMSE), which is also
adopted in this work. 

A desirable feature of a nonlinear filter is \textit{recursiveness}
in time, as it greatly reduces computational complexity and allows
for real time estimation as new measurements become available. However,
not all nonlinear filters possess this important property \cite{Segall1976,Elliott1994Hidden}.
Recursive nonlinear filters exist for some very special cases, such
as those in which the transition model of the state process is linear
(Gauss-Markov), or when the state is a Markov chain (discrete state
space) \cite{Segall_Point1976,Marcus1979,Elliott1994Exact,Elliott1994_HowToCount}.
In the absence of recursive filter representations, practical filtering
schemes have been developed, which typically approximate the desired
quantities of interest, either heuristically (e.g., Gaussian approximations
\cite{Kushner1967_Approximations,ItoXiong1997}) or in some more powerful,
rigorous sense (e.g., Markov chain approximations \cite{Kushner2001_BOOK,Kushner2008}).

In this paper, we follow the latter research direction. Specifically,
we consider a partially observable system in \textit{discrete time},
comprised by a hidden, almost surely compactly bounded state process,
observed through another, conditionally Gaussian measurement process.
The mean and covariance matrix of the measurements both constitute
nonlinear, time varying and state dependent functions, assumed to
be known apriori. Employing a change of measure argument and using
the original measurements, an approximate filtering operator can be
defined, by replacing the ``true'' state process by an appropriate
approximation. Our contribution is summarized in showing that if the
approximation converges to the state either in probability or in the
${\cal C}$-weak sense (Section II.C), the resulting filtering operator
converges to the true optimal nonlinear filter in a relatively strong
and well defined sense; the convergence is compact in time and uniform
in a measurable set of probability measure almost unity (Theorem \ref{CONVERGENCE_THEOREM}).
The aforementioned set is completely characterized in terms of a subset
of the parameters of the filtering problem of interest. Consequently,
our results provide a purely quantitative justification of Egoroff's
theorem \cite{Richardson2009measure} for the problem at hand, which
concerns the equivalence of almost sure convergence and almost uniform
convergence of measurable functions.

To better motivate the reader, let us describe two problems that fit
the scenario described above and can benefit from the contributions
of this paper, namely, those of \textit{sequential channel state estimation}
and \textit{(sequential) spatiotemporal channel prediction} \cite{KalPetChannelMarkov2014}
(see also \cite{KalPetGRID2014}). The above problems arise naturally
in novel signal processing applications in the emerging area of distributed,
autonomous, physical layer aware mobile networks \cite{KalPet-Jammers-2013,KalPet-Mobi-2014,NikosBeam-2}.
Such networks usually consist of cooperating mobile sensors, each
of them being capable of observing its communication channel (under
a flat fading assumption), relative to a reference point in the space.
In most practical scenarios, the dominant quantities characterizing
the wireless links, such as the path loss exponent and the shadowing
power, behave as stochastic processes themselves. For instance, such
behavior may be due to physical changes in the environment and also
the inherent randomness of the communication medium itself. Then,
the path loss exponent and the shadowing power can be collectively
considered as the hidden state (suggestively called the channel state)
of a partially observable system, where the channel gains measured
at each sensor can be considered as the corresponding observations.
In general, such observations are nonlinear functionals of the state.
Assuming additionally that the channel state is a Markov process,
the main results presented herein can essentially provide strong asymptotic
guarantees for approximate sequential nonlinear channel state estimation
and spatiotemporal channel prediction, enabling physical layer aware
motion planning and stochastic control. For more details, the reader
is referred to \cite{KalPetChannelMarkov2014}.

The idea of replacing the process of interest with some appropriate
approximation is borrowed from \cite{Kushner2008}. However, \cite{Kushner2008}
deals almost exclusively with continuous time stochastic systems and
the results presented in there do not automatically extend to the
discrete time system setting we are dealing with here. In fact, the
continuous time counterparts of the discrete time stochastic processes
considered here are considerably more general than the ones treated
in \cite{Kushner2008}. More specifically, although some relatively
general results are indeed provided for continuous time hidden processes,
\cite{Kushner2008} is primarily focused on the standard hidden diffusion
case, which constitutes a Markov process (and aiming to the development
of recursive approximate filters), whereas, in our setting, the hidden
process is initially assumed to be arbitrary (as long as it is confined
to a compact set). Also, different from our formulation (see above),
in \cite{Kushner2008}, the covariance matrix of the observation process
\textit{does not} depend on the hidden state; the state affects only
the mean of the observations. Further, the modes of stochastic convergence
considered here are different compared to \cite{Kushner2008} (in
fact, they are stronger), both regarding convergence of approximations
and convergence of approximate filters.

The results presented in this paper provide a framework for analyzing
a number of heuristic techniques for numerically approximating optimal
nonlinear filters in discrete time, such as approximate grid based
recursive approaches, known to perform well in a wide variety of applications
\cite{Pages2005optimal,PARTICLE2002tutorial,KalPetChannelMarkov2014}.
Additionally, our results do not refer exclusively to recursive nonlinear
filters. The sufficient conditions which we provide for the convergence
of approximate filtering operators are independent of the way a filter
is realized (see Section III). This is useful because, as highlighted
in \cite{Daum2005}, no one prevents one from designing an efficient
(approximate) nonlinear filter which is part recursive and part nonrecursive,
or even possibly trying to combine the best of both worlds, and there
are practical filters designed in this fashion \cite{Daum2005}.

The paper is organized as follows: In Section II, we introduce the
system model, along with some mild technical assumptions on its structure
and also present/develop some preliminary technical results and definitions,
which are important for stating and proving our results. In Section
III, we formulate our problem in detail and present our main results
(Theorem (\ref{CONVERGENCE_THEOREM})), along with a simple instructive
example. Section IV is exclusively devoted to proving the results
stated in Section III. Finally, Section V concludes the paper.

\textit{Notation}: In the following, the state vector will be represented
as $X_{t}$, its approximations as $X_{t}^{L_{S}}$, and all other
matrices and vectors, either random or not, will be denoted by boldface
letters (to be clear by the context). Real valued random variables
and abstract random elements will be denoted by uppercase letters.
Calligraphic letters and formal script letters will denote sets and
$\sigma$-algebras, respectively. The operators $\left(\cdot\right)^{\boldsymbol{T}}$,
$\lambda_{min}\left(\cdot\right)$, $\lambda_{max}\left(\cdot\right)$
will denote transposition, minimum and maximum eigenvalue, respectively.
For any random element (same for variable, vector) $Y$, $\sigma\left\{ Y\right\} $
will denote the $\sigma$-algebra generated by $Y$. The $\ell_{p}$
norm of a vector $\boldsymbol{x}\in\mathbb{R}^{n}$ is $\left\Vert \boldsymbol{x}\right\Vert _{p}\triangleq\left(\sum_{i=1}^{n}\left|x_{i}\right|^{p}\right)^{1/p}$,
for all naturals $p\ge1$. The spectral and Frobenius norms of any
matrix ${\bf X}\in\mathbb{R}^{n\times n}$ are $\left\Vert {\bf X}\right\Vert _{2}\triangleq\max_{\left\Vert \boldsymbol{x}\right\Vert _{2}\equiv1}\left\Vert {\bf X}\boldsymbol{x}\right\Vert _{2}$
and $\left\Vert {\bf X}\right\Vert _{F}\triangleq\sqrt{\sum_{i,j=1}^{n,n}\left|{\bf X}_{ij}\right|^{2}}$,
respectively. Positive definiteness and semidefiniteness of ${\bf X}$
will be denoted by ${\bf X}\succ{\bf 0}$ and ${\bf X}\succeq{\bf 0}$,
respectively. For any Euclidean space $\mathbb{R}^{N\times1}$, ${\bf I}_{N\times N}$
will denote the respective identity operator. Additionally, throughout
the paper, we employ the identifications $\mathbb{R}_{+}\equiv\left[0,\infty\right)$,
$\mathbb{R}_{++}\equiv\left(0,\infty\right)$, $\mathbb{N}^{+}\equiv\left\{ 1,2,\ldots\right\} $,
$\mathbb{N}_{n}^{+}\equiv\left\{ 1,2,\ldots,n\right\} $ and $\mathbb{N}_{n}\equiv\left\{ 0\right\} \cup\mathbb{N}_{n}^{+}$,
for any positive natural $n$.

\section{Partially Observable System Model \& Technical Preliminaries}

In this section, we give a detailed description of the partially observable
(or hidden) system model of interest and present our related technical
assumptions on its components. Additionally, we present some essential
background on the measure theoretic concept of change of probability
measures and state some definitions and known results regarding specific
modes of stochastic convergence, which will be employed in our subsequent
theoretical developments.

\subsection{Hidden Model: Definitions \& Technical Assumptions}

First, let us set the basic probabilistic framework, as well as precisely
define the hidden system model considered throughout the paper:
\begin{itemize}
\item All stochastic processes considered below are fundamentally generated
on a common complete probability space (the base space), defined by
a triplet $\left(\Omega,\mathscr{F},{\cal P}\right)$, at each time
instant taking values in a measurable state space, consisting of some
Euclidean subspace and the associated Borel $\sigma$-algebra on that
subspace. For example, for each $t\in\mathbb{N}$, the state process
$X_{t}\equiv X_{t}\left(\omega\right)$, where $\omega\in\Omega$,
takes its values in the measurable state space $\left(\mathbb{R}^{M\times1},\mathscr{B}\left(\mathbb{R}^{M\times1}\right)\right)$,
where $\mathscr{B}\left(\mathbb{R}^{M\times1}\right)$ constitutes
the Borel $\sigma$-algebra of measurable subsets of $\mathbb{R}^{M\times1}$.
\item In this work, the evolution mechanism of state process $X_{t}$ is
assumed to be arbitrary. However, in order to avoid unnecessary technical
complications, we assume that, for each $t\in\mathbb{N}$, the induced
probability measure of $X_{t}$ is absolutely continuous with respect
to the Lebesgue measure on its respective state space. Then, by the
Radon-Nikodym Theorem, it admits a density, unique up to sets of zero
Lebesgue measure. Also, we will generically assume that for all $t\in\mathbb{N}$,
$X_{t}\in{\cal Z}$ almost surely, where ${\cal Z}$ constitutes a
compact strict subset of $\mathbb{R}^{M\times1}$. In what follows,
however, in order to lighten the presentation, we will assume that
$M\equiv1$. Nevertheless, all stated results hold with the same validity
if $M>1$ (See also Assumption 2 below). 
\item The state $X_{t}$ is partially observed through the observation process
\begin{equation}
{\bf y}_{t}\triangleq\boldsymbol{\mu}_{t}\left(X_{t}\right)+\boldsymbol{\sigma}_{t}\left(X_{t}\right)+\boldsymbol{\xi}_{t}\in\mathbb{R}^{N\times1},\quad\forall t\in\mathbb{N},\label{eq:Observation_Equation}
\end{equation}
where, \textit{conditioned on $X_{t}$} and for each $t\in\mathbb{N}$,
the sequence $\left\{ \boldsymbol{\mu}_{t}:{\cal Z}\mapsto\mathbb{R}^{N\times1}\right\} _{t\in\mathbb{N}}$
is known apriori, the process $\boldsymbol{\sigma}_{t}\left(X_{t}\right)\sim{\cal N}\left({\bf 0},\boldsymbol{\Sigma}_{t}\left(X_{t}\right)\succ{\bf 0}\right)$
constitutes Gaussian noise, with the sequence $\left\{ \boldsymbol{\Sigma}_{t}:{\cal Z}\mapsto{\cal D}_{\boldsymbol{\Sigma}}\right\} _{t\in\mathbb{N}}$,
where ${\cal D}_{\boldsymbol{\Sigma}}$ is a bounded subset of $\mathbb{R}^{N\times N}$,
also known apriori, and $\boldsymbol{\xi}_{t}\overset{i.i.d.}{\sim}{\cal N}\left({\bf 0},\sigma_{\xi}^{2}{\bf I}_{N\times N}\right)$. 
\end{itemize}
As a pair, the state $X_{t}$ and the observations process described
by (\ref{eq:Observation_Equation}) define a very wide family of partially
observable systems. In particular, any Hidden Markov Model (HMM) of
\textit{any} order, in which the respective Markov state process is
almost surely confined in a compact subset of its respective Euclidean
state space, is indeed a member of this family. More specifically,
let us rewrite (\ref{eq:Observation_Equation}) in the canonical form
\begin{equation}
{\bf y}_{t}\equiv\boldsymbol{\mu}_{t}\left(X_{t}\right)+\sqrt{{\bf C}_{t}\left(X_{t}\right)}\boldsymbol{u}_{t}\in\mathbb{R}^{N\times1},\quad\forall t\in\mathbb{N},
\end{equation}
where $\boldsymbol{u}_{t}\equiv\boldsymbol{u}_{t}\left(\omega\right)$
constitutes a standard Gaussian white noise process and, for all $x\in{\cal Z}$,
${\bf C}_{t}\left(x\right)\triangleq\boldsymbol{\Sigma}_{t}\left(x\right)+\sigma_{\xi}^{2}{\bf I}_{N\times N}\in{\cal D}_{{\bf C}}$,
with ${\cal D}_{{\bf C}}$ bounded. Then, for a possibly nonstationary
HMM of order $m$, assuming the existence of an explicit functional
model for describing the temporal evolution of the state (being a
Markov process of order $m$), we get the system of \textit{standardized}
stochastic difference equations\renewcommand{\arraystretch}{1.5}
\begin{equation}
\begin{array}{l}
X_{t}\equiv f_{t}\left(\left\{ X_{t-i}\right\} _{i\in\mathbb{N}_{m}^{+}},\boldsymbol{W}_{t}\right)\in{\cal Z}\\
\,{\bf y}_{t}\hspace{0.5bp}\equiv\boldsymbol{\mu}_{t}\left(X_{t}\right)+\sqrt{{\bf C}_{t}\left(X_{t}\right)}\boldsymbol{u}_{t}
\end{array},\quad\forall t\in\mathbb{N},
\end{equation}
\renewcommand{\arraystretch}{1.5}where, for each $t$, $f_{t}:{\cal Z}^{m}\times{\cal W}\overset{a.s.}{\mapsto}{\cal Z}$
(with ${\cal Z}^{m}\triangleq\times_{m\text{ times }}{\cal Z}$) constitutes
a measurable nonlinear state transition mapping and $\boldsymbol{W}_{t}\equiv\boldsymbol{W}_{t}\left(\omega\right)\in{\cal W}\subseteq\mathbb{R}^{M_{W}\times1}$,
denotes a (discrete time) white noise process with state space ${\cal W}$.
For a first order stationary HMM, the above system of equations reduces
to
\begin{equation}
\begin{array}{l}
X_{t}\equiv f\left(X_{t-1},\boldsymbol{W}_{t}\right)\in{\cal Z}\\
\,{\bf y}_{t}\hspace{0.5bp}\equiv\boldsymbol{\mu}_{t}\left(X_{t}\right)+\sqrt{{\bf C}_{t}\left(X_{t}\right)}\boldsymbol{u}_{t}
\end{array},\quad\forall t\in\mathbb{N},
\end{equation}
which arguably constitutes the most typical partially observable system
model encountered in both Signal Processing and Control, with plethora
of important applications.

Let us also present some more specific assumptions, regarding the
nature (boundedness, continuity and expansiveness) of the aforementioned
sequences of functions.\textbf{\medskip{}
}

\noindent \textbf{Assumption 1:} \label{Elementary_Definition_1-1}\textbf{(Boundedness)}
For later reference, let
\begin{flalign}
\lambda_{inf} & \triangleq{\displaystyle \inf_{t\in\mathbb{N}}}\hspace{0.2em}{\displaystyle \inf_{x\in{\cal Z}}}\hspace{0.2em}\lambda_{min}\left({\bf C}_{t}\left(x\right)\right),\\
\lambda_{sup} & \triangleq{\displaystyle \sup_{t\in\mathbb{N}}}\hspace{0.2em}{\displaystyle \sup_{x\in{\cal Z}}}\hspace{0.2em}\lambda_{max}\left({\bf C}_{t}\left(x\right)\right),\\
\mu_{sup} & \triangleq{\displaystyle \sup_{t\in\mathbb{N}}}\hspace{0.2em}{\displaystyle \sup_{x\in{\cal Z}}}\hspace{0.2em}\left\Vert \boldsymbol{\mu}_{t}\left(x\right)\right\Vert _{2},
\end{flalign}
where each quantity of the above is uniformly and finitely bounded
for all $t\in\mathbb{N}$ and for all $x\in{\cal Z}$. If $x$ is
substituted by the stochastic process $X_{t}\left(\omega\right)$,
then all the above definitions continue to hold in the essential sense.
For technical reasons related to the bounding-from-above arguments
presented in Section IV, containing the proof of the main result of
the paper, it is also assumed that $\lambda_{inf}>1$, a requirement
which can always be satisfied by appropriate normalization of the
observations.\textbf{\medskip{}
}

\noindent \textbf{Assumption 2:} \textbf{(Continuity \& Expansiveness)}
All members of the functional family\linebreak{}
$\left\{ \boldsymbol{\mu}_{t}:{\cal Z}\mapsto\mathbb{R}^{N\times1}\right\} _{t\in\mbox{\ensuremath{\mathbb{N}}}}$
are uniformly Lipschitz continuous, that is, there exists a bounded
constant $K_{\mu}\in\mathbb{R}_{+}$, such that, for all $t\in\mathbb{N}$,
\begin{equation}
\left\Vert \boldsymbol{\mu}_{t}\left(x\right)-\boldsymbol{\mu}_{t}\left(y\right)\right\Vert _{2}\le K_{\mu}\left|x-y\right|,\quad\forall\left(x,y\right)\in{\cal Z}\times{\cal Z}.
\end{equation}
Additionally, all members of the functional family $\left\{ \boldsymbol{\Sigma}_{t}:{\cal Z}\mapsto{\cal D}_{\boldsymbol{\Sigma}}\subset\mathbb{R}^{N\times N}\right\} _{t\in\mathbb{N}}$
are \textit{elementwise} uniformly Lipschitz continuous, that is,
there exists some universal and bounded constant $K_{\boldsymbol{\Sigma}}\in\mathbb{R}_{+}$,
such that, for all $t\in\mathbb{N}$ and for all $\left(i,j\right)\in\mathbb{N}_{N}^{+}\times\mathbb{N}_{N}^{+}$,
\begin{equation}
\left|\boldsymbol{\Sigma}_{t}^{ij}\left(x\right)-\boldsymbol{\Sigma}_{t}^{ij}\left(y\right)\right|\le K_{\boldsymbol{\Sigma}}\left|x-y\right|,\quad\forall\left(x,y\right)\in{\cal Z}\times{\cal Z}.
\end{equation}
If $x$ is substituted by the stochastic process $X_{t}\left(\omega\right)$,
then all the above statements are understood in the almost sure sense.
\begin{rem}
As we have already said, for simplicity, we assume that ${\cal Z}\subset\mathbb{R}$,
that is, $M\equiv1$. In any other case (when $M>1$), we modify the
Lipschitz assumptions stated above simply by replacing $\left|x-y\right|$
with $\left\Vert {\bf x}-{\bf y}\right\Vert _{1}$, that is, Lipschitz
continuity is meant to be with respect to the $\ell_{1}$ norm in
the domain of the respective function. If this holds, everything that
follows works also in $\mathbb{R}^{M>1}$, just with some added complexity
in the proofs of the results. Also, because $\left\Vert {\bf x}\right\Vert _{2}\le\left\Vert {\bf x}\right\Vert _{1}$
for any ${\bf x}\in\mathbb{R}^{M}$, the assumed Lipschitz continuity
with respect to $\ell_{1}$ norm can be replaced by Lipschitz continuity
with respect to the $\ell_{2}$ norm, since the latter implies the
former, and again everything holds. Further, if $M>1$, convergence
in probability and ${\cal L}_{1}$ convergence of random vectors are
both defined by replacing absolute values with the $\ell_{1}$ norms
of the vectors under consideration.\hfill{}\ensuremath{\blacksquare}
\end{rem}

\subsection{Conditional Expectations, Change of Measure \& Filters}

Before proceeding with the general formulation of our estimation problem
and for later reference, let us define the complete natural filtrations
of the processes $X_{t}$ and ${\bf y}_{t}$ as
\begin{flalign}
\left\{ \mathscr{X}_{t}\right\} _{t\in\mathbb{N}} & \triangleq\left\{ \sigma\left\{ \left\{ X_{i}\right\} _{i\in\mathbb{N}_{t}}\right\} \right\} _{t\in\mathbb{N}}\quad\text{and}\\
\left\{ \mathscr{Y}_{t}\right\} _{t\in\mathbb{N}} & \triangleq\left\{ \sigma\left\{ \left\{ {\bf y}_{i}\right\} _{i\in\mathbb{N}_{t}}\right\} \right\} _{t\in\mathbb{N}},
\end{flalign}
respectively, and also the complete filtration generated by both $X_{t}$
and ${\bf y}_{t}$ as
\begin{equation}
\left\{ \mathscr{H}_{t}\right\} _{t\in\mathbb{N}}\triangleq\left\{ \sigma\left\{ \left\{ X_{i},{\bf y}_{i}\right\} _{i\in\mathbb{N}_{t}}\right\} \right\} _{t\in\mathbb{N}}.
\end{equation}
In all the above, $\sigma\left\{ Y\right\} $ denotes the $\sigma$-algebra
generated by the random variable $Y$.

In this work, we adopt the MMSE as an optimality criterion. In this
case, one would ideally like to discover a solution to the stochastic
optimization problem\renewcommand{\arraystretch}{1.5}
\begin{equation}
\begin{array}{rl}
\inf_{\widehat{X}_{t}} & \mathbb{E}\left\{ \left\Vert X_{t}-\widehat{X}_{t}\right\Vert _{2}^{2}\right\} \\
\mathrm{subject\,to} & \mathbb{E}\left\{ \left.\widehat{X}_{t}\right|\mathscr{Y}_{t}\right\} \equiv\widehat{X}_{t}
\end{array},\quad\forall t\in\mathbb{N},\label{eq:MMSE_Optimization}
\end{equation}
\renewcommand{\arraystretch}{1}where the constraint is equivalent
to confining the search for possible estimators $\widehat{X}_{t}$
to the subset of interest, that is, containing the ones which constitute
$\mathscr{Y}_{t}$-measurable random variables. Of course, the solution
to the program (\ref{eq:MMSE_Optimization}) coincides with the conditional
expectation \cite{Papoulis2002}
\begin{equation}
\mathbb{E}\left\{ \left.X_{t}\right|\mathscr{Y}_{t}\right\} \equiv\widehat{X}_{t},\quad\forall t\in\mathbb{N},\label{eq:Cond_Expectation}
\end{equation}
which, in the nonlinear filtering literature, is frequently called
a \textit{filter}. There is also an alternative and very useful way
of reexpressing the filter process $\widehat{X}_{t}$, using the concept
of \textit{change of probability measures}, which will allow us to
stochastically decouple the state and observations of our hidden system
and then let us formulate precisely the approximation problem of interest
in this paper. Change of measure techniques have been extensively
used in discrete time nonlinear filtering, mainly in order to discover
recursive representations for various hidden Markov models \cite{Elliott1994Hidden,Elliott1994Exact,Elliott1994_HowToCount,ElliotKrishnamurthy1999}.
In the following, we provide a brief introduction to these type of
techniques (suited to our purposes) which is also intuitive, simple
and technically accessible, including direct proofs of the required
results.\textbf{\textit{\medskip{}
}}

\noindent \textbf{\textit{Change of Probability Measure in Discrete
Time: Demystification \& Useful Results\medskip{}
}}

So far, all stochastic processes we have considered are defined on
the base space $\left(\Omega,\mathscr{F},{\cal P}\right)$. In fact,
it is the structure of the probability measure ${\cal P}$ that is
responsible for the coupling between the stochastic processes $X_{t}$
and ${\bf y}_{t}$, being, for each $t\in\mathbb{N}$, measurable
functions from $\left(\Omega,\mathscr{F}\right)$ to $\left(\mathbb{R},\mathscr{B}\left(\mathbb{R}\right)\right)$
and $\left(\mathbb{R}^{N},\mathscr{B}\left(\mathbb{R}^{N}\right)\right)$,
respectively. Intuitively, the measure ${\cal P}$ constitutes our
``reference measurement tool'' for measuring the events contained
in the base $\sigma$-algebra $\mathscr{F}$, and any random variable
serves as a ``medium'' or ``channel'' for observing these events.

As a result, some very natural questions arise from the above discussion.
First, one could ask if and under what conditions it is possible to
change the probability measure ${\cal P}$, which constitutes our
\textit{fixed} way of assigning probabilities to events, to another
measure $\widetilde{{\cal P}}$ on the same measurable space $\left(\Omega,\mathscr{F}\right)$,
in a way such that there exists some sort of transformation connecting
${\cal P}$ and $\widetilde{{\cal P}}$. Second, if we can indeed
make the transition from ${\cal P}$ to $\widetilde{{\cal P}}$, could
we choose the latter probability measure in a way such that the processes
$X_{t}$ and ${\bf y}_{t}$ behave according to a prespecified statistical
model? For instance, we could demand that, under $\widetilde{{\cal P}}$,
$X_{t}$ and ${\bf y}_{t}$ constitute independent stochastic processes.
Third and most important, is it possible to derive an expression for
the ``original'' filter $\widehat{X}_{t}\equiv\mathbb{E}_{{\cal P}}\left\{ \left.X_{t}\right|\mathscr{Y}_{t}\right\} $
under measure ${\cal P}$, using only (conditional) expectations under
$\widetilde{{\cal P}}$ (denoted as $\mathbb{E}_{\widetilde{{\cal P}}}\left\{ \left.\cdot\right|\cdot\right\} $)?

The answers to all three questions stated above are affirmative under
very mild assumptions and the key result in order to prove this assertion
is the Radon-Nikodym Theorem \cite{Ash2000Probability}. However,
assuming that the induced joint probability measure of the processes
of interest is absolutely continuous with respect to the Lebesgue
measure of the appropriate dimension, in the following we provide
an answer to these questions, employing only elementary probability
theory, avoiding the direct use of the Radon-Nikodym Theorem.
\begin{thm}
\textbf{\textup{\label{thm:(Conditional-Bayes'-Theorem)}(Conditional
Bayes' Theorem for Densities)}} Consider the (possibly vector) stochastic
processes $X_{t}\left(\omega\right)\in\mathbb{R}^{N_{t}\times1}$
and $Y_{t}\left(\omega\right)\in\mathbb{R}^{M_{t}\times1}$, both
defined on the same measurable space $\left(\Omega,\mathscr{F}\right)$,
for all $t\in\mathbb{N}$. Further, if ${\cal P}$ and $\widetilde{{\cal P}}$
are two probability measures on $\left(\Omega,\mathscr{F}\right)$,
suppose that:
\begin{itemize}
\item Under both ${\cal P}$ and $\widetilde{{\cal P}}$, the process $X_{t}$
is integrable.
\item Under the base probability measure ${\cal P}$ (resp. $\widetilde{{\cal P}}$),
the induced joint probability measure of\linebreak{}
$\left(\left\{ X_{i}\right\} _{i\in\mathbb{N}_{t}},\left\{ Y_{i}\right\} _{i\in\mathbb{N}_{t}}\right)$
is absolutely continuous with respect to the Lebesgue measure of the
appropriate dimension, implying the existence of a density $f_{t}$
(resp. $\widetilde{f}_{t}$), with
\begin{equation}
f_{t}:\left(\underset{i\in\mathbb{N}_{t}}{\times}\mathbb{R}^{N_{i}\times1}\right)\times\left(\underset{i\in\mathbb{N}_{t}}{\times}\mathbb{R}^{M_{i}\times1}\right)\mapsto\mathbb{R}_{+}.
\end{equation}

\item For each set of points, it is true that
\begin{equation}
\widetilde{f}_{t}\left(\cdots\right)\equiv0\quad\Rightarrow\quad f_{t}\left(\cdots\right)\equiv0,\label{eq:ABSOLUTE_Densities}
\end{equation}
or, equivalently, the support of $f_{t}$ is contained in the support
of $\widetilde{f}_{t}$.
\end{itemize}
Also, for all $t\in\mathbb{N}$, define the Likelihood Ratio (LR)
at $t$ as the $\left\{ \mathscr{H}_{t}\right\} $-adapted, nonnegative
stochastic process\footnote{With zero probability of confusion, we use $\left\{ \mathscr{Y}_{t}\right\} _{t\in\mathbb{N}}$
and $\left\{ \mathscr{H}_{t}\right\} _{t\in\mathbb{N}}$ to denote
the complete filtrations generated by $Y_{t}$ and $\left\{ X_{t},Y_{t}\right\} $.\textit{ }}
\begin{equation}
\Lambda_{t}\triangleq\dfrac{f_{t}\left(X_{0},X_{1},\ldots,X_{t},Y_{0},Y_{1},\ldots,Y_{t}\right)}{\widetilde{f}_{t}\left(X_{0},X_{1},\ldots,X_{t},Y_{0},Y_{1},\ldots,Y_{t}\right)}.\label{eq:Likelihood_Ratio}
\end{equation}
Then, it is true that
\begin{equation}
\widehat{X}_{t}\equiv\mathbb{E}_{{\cal P}}\left\{ \left.X_{t}\right|\mathscr{Y}_{t}\right\} \equiv\dfrac{\mathbb{E}_{\widetilde{{\cal P}}}\left\{ \left.X_{t}\Lambda_{t}\right|\mathscr{Y}_{t}\right\} }{\mathbb{E}_{\widetilde{{\cal P}}}\left\{ \left.\Lambda_{t}\right|\mathscr{Y}_{t}\right\} },\label{eq:Changed_Expectation}
\end{equation}
almost everywhere with respect to ${\cal P}$.\end{thm}
\begin{proof}[Proof of Theorem \ref{thm:(Conditional-Bayes'-Theorem)}]
 See the Appendix.\end{proof}
\begin{rem}
The $\left\{ \mathscr{H}_{t}\right\} $-adapted LR process 
\begin{equation}
\Lambda_{t}\equiv\Lambda_{t}\left({\cal X}_{t}\triangleq\left\{ X_{i}\right\} _{i\in\mathbb{N}_{t}},{\cal Y}_{t}\triangleq\left\{ Y_{i}\right\} _{i\in\mathbb{N}_{t}}\right),\quad t\in\mathbb{N},\label{eq:RDLR}
\end{equation}
as defined in (\ref{eq:Likelihood_Ratio}), actually coincides with
the restriction of the Radon-Nikodym derivative of ${\cal P}$ with
respect to $\widetilde{{\cal P}}$ to the filtration $\left\{ \mathscr{H}_{t}\right\} _{t\in\mathbb{N}}$,
that is,
\begin{equation}
\left.\dfrac{\mathrm{d}{\cal P}\left(\omega\right)}{\mathrm{d}\widetilde{{\cal P}}\left(\omega\right)}\right|_{\mathscr{H}_{t}}\equiv\Lambda_{t}\left({\cal X}_{t}\left(\omega\right),{\cal Y}_{t}\left(\omega\right)\right),\quad\forall t\in\mathbb{N},
\end{equation}
a statement which, denoting the collections $\left\{ x_{i}\right\} _{i\in\mathbb{N}_{t}}$
and $\left\{ y_{i}\right\} _{i\in\mathbb{N}_{t}}$ as $\mathpzc{x}_{t}$
and $\mathpzc{y}_{t}$, respectively, is rigorously equivalent to
\begin{flalign}
{\cal P}\left({\cal F}\right) & \equiv\int_{{\cal F}}\Lambda_{t}\left({\cal X}_{t}\left(\omega\right),{\cal Y}_{t}\left(\omega\right)\right)\mathrm{d}\widetilde{{\cal P}}\left(\omega\right)\nonumber \\
 & \equiv\mathop{\mathlarger{\mathlarger{\int_{{\cal B}}}}}\Lambda_{t}\left(\mathpzc{x}_{t},\mathpzc{y}_{t}\right)\mathrm{d}^{2t}\widetilde{{\cal P}}_{\left({\cal X}_{t},{\cal Y}_{t}\right)}\left(\mathpzc{x}_{t},\mathpzc{y}_{t}\right)\nonumber \\
 & \equiv{\cal P}_{\left({\cal X}_{t},{\cal Y}_{t}\right)}\left({\cal B}\right)\equiv{\cal P}\left(\left({\cal X}_{t},{\cal Y}_{t}\right)\in{\cal B}\right),\\
\\
\forall{\cal F} & \triangleq\left\{ \omega\in\Omega\left|\left({\cal X}_{t}\left(\omega\right),{\cal Y}_{t}\left(\omega\right)\right)\in{\cal B}\right.\right\} \in\mathscr{H}_{t}\quad\text{and}\\
\forall{\cal B} & \in\left(\underset{i\in\mathbb{N}_{t}}{\otimes}\mathscr{B}\left(\mathbb{R}^{N_{i}\times1}\right)\right)\otimes\left(\underset{i\in\mathbb{N}_{t}}{\otimes}\mathscr{B}\left(\mathbb{R}^{M_{i}\times1}\right)\right),\quad\forall t\in\mathbb{N},
\end{flalign}
respectively (in the above, ``$\otimes$'' denotes the product operator
for $\sigma$-algebras). Of course, the existence and almost everywhere
uniqueness of $\Lambda_{t}$ are guaranteed by the Radon-Nikodym Theorem,
provided that the base measure ${\cal P}$ is absolutely continuous
with respect to $\widetilde{{\cal P}}$ on $\mathscr{H}_{t}$ (${\cal P}\ll_{\mathscr{H}_{t}}\widetilde{{\cal P}}$).
Further, for the case where there exist densities characterizing ${\cal P}$
and $\widetilde{{\cal P}}$ (as in Theorem 1), demanding that ${\cal P}\ll_{\mathscr{H}_{t}}\widetilde{{\cal P}}$
is precisely equivalent to demanding that (\ref{eq:ABSOLUTE_Densities})
is true and, again through the Radon-Nikodym Theorem, it can be easily
shown that the derivative $\Lambda_{t}$ actually coincides with the
likelihood ratio process defined in (\ref{eq:Likelihood_Ratio}),
almost everywhere.\hfill{}\ensuremath{\blacksquare}
\end{rem}
Now, let us apply Theorem 1 for the stochastic processes $X_{t}$
and ${\bf y}_{t}$, comprising our partially observed system, as defined
in Section II.A. In this respect, we present the following result.
\begin{thm}
\label{thm:COMCOM}\textbf{\textup{(Change of Measure for the Hidden
System under Study)}} Consider the hidden stochastic system of Section
II.A on the usual base space $\left(\Omega,\mathscr{F},{\cal P}\right)$,
where $X_{t}\in{\cal Z}$ and ${\bf y}_{t}\in\mathbb{R}^{N\times1}$,
almost surely $\forall t\in\mathbb{N}$, constitute the hidden state
process and the observation process, respectively. Then, there exists
an alternative, \textbf{equivalent} to\textup{ ${\cal P}$,} base
measure \textup{$\widetilde{{\cal P}}$ on }$\left(\Omega,\mathscr{F}\right)$\textup{,
}under which:
\begin{itemize}
\item The processes $X_{t}$ and ${\bf y}_{t}$ are statistically independent.
\item $X_{t}$ constitutes a stochastic process with exactly the same dynamics
as under ${\cal P}$.
\item ${\bf y}_{t}$ constitutes a Gaussian vector white noise process with
zero mean and covariance matrix equal to the identity.
\end{itemize}
Additionally, the filter $\widehat{X}_{t}$ can be expressed as in
(\ref{eq:Changed_Expectation}), where the $\left\{ \mathscr{H}_{t}\right\} $-adapted
stochastic process $\Lambda_{t},t\in\mathbb{N}$ is defined as in
(\ref{eq:LR_SPECIFIC}) (top of next page).
\end{thm}
\noindent \begin{figure*}[t!]
\hrulefill
\normalsize

\noindent \textbf{\textit{\vspace{-0.8cm}
}}

\begin{flalign}
\Lambda_{t} & \triangleq{\displaystyle \prod_{i\in\mathbb{N}_{t}}}\dfrac{\exp\left(\dfrac{1}{2}\left\Vert {\bf y}_{i}\right\Vert _{2}^{2}-\dfrac{1}{2}\left({\bf y}_{i}-\boldsymbol{\mu}_{i}\left(X_{i}\right)\right)^{\boldsymbol{T}}\left(\boldsymbol{\Sigma}_{i}\left(X_{i}\right)+\sigma_{\xi}^{2}{\bf I}_{N\times N}\right)^{-1}\left({\bf y}_{i}-\boldsymbol{\mu}_{i}\left(X_{i}\right)\right)\right)}{\sqrt{\det\left(\boldsymbol{\Sigma}_{i}\left(X_{i}\right)+\sigma_{\xi}^{2}{\bf I}_{N\times N}\right)}}\triangleq{\displaystyle \prod_{i\in\mathbb{N}_{t}}}\lambda_{i}\nonumber \\
 & \equiv\dfrac{\exp\left(\dfrac{1}{2}{\displaystyle \sum_{i\in\mathbb{N}_{t}}}\left\Vert {\bf y}_{i}\right\Vert _{2}^{2}-\left({\bf y}_{i}-\boldsymbol{\mu}_{i}\left(X_{i}\right)\right)^{\boldsymbol{T}}\left(\boldsymbol{\Sigma}_{i}\left(X_{i}\right)+\sigma_{\xi}^{2}{\bf I}_{N\times N}\right)^{-1}\left({\bf y}_{i}-\boldsymbol{\mu}_{i}\left(X_{i}\right)\right)\right)}{{\displaystyle \prod_{i\in\mathbb{N}_{t}}}\sqrt{\det\left(\boldsymbol{\Sigma}_{i}\left(X_{i}\right)+\sigma_{\xi}^{2}{\bf I}_{N\times N}\right)}}\in\mathbb{R}_{++}\label{eq:LR_SPECIFIC}
\end{flalign}
\hrulefill
\vspace*{-30pt}
\end{figure*}
\begin{proof}[Proof of Theorem 2]
Additionally to the similar identifications made above (see (\ref{eq:RDLR}))
and for later reference, let
\begin{equation}
\boldsymbol{{\cal Y}}_{t}\triangleq\left\{ {\bf y}_{i}\right\} _{i\in\mathbb{N}_{t}}\quad\text{and}\quad\mathpzc{y}_{t}\triangleq\left\{ \boldsymbol{y}_{i}\right\} _{i\in\mathbb{N}_{t}}.\label{eq:X_Cal_2}
\end{equation}
First, we construct the probability measure $\widetilde{{\cal P}}$,
this way showing its existence. To accomplish this, define, for each
$t\in\mathbb{N}$, a probability measure $\widetilde{{\cal P}}_{{\cal R}_{t}}$
on the measurable space $\left({\cal R}_{t},\mathscr{B}\left({\cal R}_{t}\right)\right)$,
where
\begin{equation}
{\cal R}_{t}\triangleq\left(\underset{i\in\mathbb{N}_{t}}{\times}\mathbb{R}\right)\times\left(\underset{i\in\mathbb{N}_{t}}{\times}\mathbb{R}^{N\times1}\right),
\end{equation}
being absolutely continuous with respect to the Lebesgue measure on
$\left({\cal R}_{t},\mathscr{B}\left({\cal R}_{t}\right)\right)$
and with density $\widetilde{f}_{t}:{\cal R}\mapsto\mathbb{R}_{+}$.
Since, for each $t\in\mathbb{N}$, the processes $X_{t}\left(\omega\right)$
and ${\bf y}_{t}\left(\omega\right)$ are both, by definition, \textit{fixed}
and \textit{measurable} functions from $\left(\Omega,\mathscr{H}_{t}\right)$
to $\left({\cal R}_{t},\mathscr{B}\left({\cal R}_{t}\right)\right)$,
with\footnote{$\mathscr{H}_{\infty}$ constitutes the \textit{join}, that is, the
\textit{smallest} $\sigma$-algebra generated by the union of all
$\mathscr{H}_{t},\forall t\in\mathbb{N}$.} 
\begin{equation}
\mathscr{H}_{t}\subseteq\mathscr{H}_{\infty}\triangleq\sigma\left\{ {\displaystyle \bigcup_{t\in\mathbb{N}}}\mathscr{H}_{t}\right\} \subseteq\mathscr{F},
\end{equation}
measuring any ${\cal B}\in\mathscr{B}\left({\cal R}_{t}\right)$ under
$\widetilde{{\cal P}}_{{\cal R}_{t}}$ can be replaced by measuring
the event (preimage)\linebreak{}
$\left\{ \left.\omega\in\Omega\right|\left({\cal X}_{t},\boldsymbol{{\cal Y}}_{t}\right)\in{\cal B}\right\} \in\mathscr{H}_{t}$
under another measure, say $\widetilde{{\cal P}}$, defined collectively
for all $t\in\mathbb{N}$ on the general measurable space $\left(\Omega,\mathscr{H}_{\infty}\right)$
as
\begin{flalign*}
\widetilde{{\cal P}}\left(\left\{ \left.\omega\in\Omega\right|\left({\cal X}_{t},\boldsymbol{{\cal Y}}_{t}\right)\in{\cal B}\right\} \right) & \equiv\widetilde{{\cal P}}\left(\left({\cal X}_{t},\boldsymbol{{\cal Y}}_{t}\right)\in{\cal B}\right)\triangleq\widetilde{{\cal P}}_{{\cal R}_{t}}\left({\cal B}\right),\quad\forall{\cal B}\in\mathscr{B}\left({\cal R}_{t}\right).
\end{flalign*}
That is, the restriction of the probability measure $\widetilde{{\cal P}}$
to the $\sigma$-algebra $\mathscr{H}_{\infty}$ is \textit{induced}
by the probability measure $\widetilde{{\cal P}}_{{\cal R}_{\infty}}$
(also see Kolmogorov's Extension Theorem \cite{Elliott1994Hidden}).
Further, in order to define the alternative base measure $\widetilde{{\cal P}}$
fully on $\left(\Omega,\mathscr{F}\right)$, we have to extend its
behavior on the remaining events which belong to the potentially finer
$\sigma$-algebra $\mathscr{F}$ but are not included in $\mathscr{H}_{\infty}$.
However, since we are interested in change of measure only for the
augmented process $\left({\cal X}_{t},\boldsymbol{{\cal Y}}_{t}\right)$,
these events are irrelevant to us. Therefore, $\widetilde{{\cal P}}$
can be defined arbitrarily on these events, as long as it remains
a valid and consistent probability measure.

Now, to finalize the construction of the restriction of $\widetilde{{\cal P}}$
to $\mathscr{H}_{t},\forall t\in\mathbb{N}$, we have to explicitly
specify the density of $\widetilde{{\cal P}}_{{\cal R}_{t}}$, or,
equivalently, of the joint density of the random variables $\left({\cal X}_{t},\boldsymbol{{\cal Y}}_{t}\right)$,
$\widetilde{f}_{t}$, for all $t\in\mathbb{N}$. According to the
statement of Theorem 2, we have to demand that
\begin{flalign}
\widetilde{f}_{t}\left(\mathpzc{x}_{t},\mathpzc{y}_{t}\right) & \equiv\widetilde{f}_{\left.\boldsymbol{{\cal Y}}_{t}\right|{\cal X}_{t}}\left(\left.\mathpzc{y}_{t}\right|\mathpzc{x}_{t}\right)\widetilde{f}_{{\cal X}_{t}}\left(\mathpzc{x}_{t}\right)\nonumber \\
 & =\widetilde{f}_{\boldsymbol{{\cal Y}}_{t}}\left(\mathpzc{y}_{t}\right)f_{{\cal X}_{t}}\left(\mathpzc{x}_{t}\right)\nonumber \\
 & =\left(\prod_{i\in\mathbb{N}_{t}}\widetilde{f}_{{\bf y}_{i}}\left(\boldsymbol{y}_{i}\right)\right)f_{{\cal X}_{t}}\left(\mathpzc{x}_{t}\right)\nonumber \\
 & =\left(\prod_{i\in\mathbb{N}_{t}}\dfrac{\exp\left(-\dfrac{\left\Vert \boldsymbol{y}_{i}\right\Vert _{2}^{2}}{2}\right)}{\sqrt{\left(2\pi\right)^{N}}}\right)f_{{\cal X}_{t}}\left(\mathpzc{x}_{t}\right)\nonumber \\
 & \equiv\dfrac{\exp\left(-\dfrac{1}{2}{\displaystyle \sum_{i\in\mathbb{N}_{t}}}\left\Vert \boldsymbol{y}_{i}\right\Vert _{2}^{2}\right)}{\sqrt{\left(2\pi\right)^{N\left(t+1\right)}}}f_{{\cal X}_{t}}\left(\mathpzc{x}_{t}\right).\label{eq:f_tilde}
\end{flalign}

Next, by definition, we know that, under ${\cal P}$, the joint density
of $\left({\cal X}_{t},\boldsymbol{{\cal Y}}_{t}\right)$ can be expressed
as
\begin{flalign}
f_{t}\left(\mathpzc{x}_{t},\mathpzc{y}_{t}\right) & \equiv f_{\left.\boldsymbol{{\cal Y}}_{t}\right|{\cal X}_{t}}\left(\left.\mathpzc{y}_{t}\right|\mathpzc{x}_{t}\right)f_{{\cal X}_{t}}\left(\mathpzc{x}_{t}\right)\nonumber \\
 & \equiv\left(\prod_{i\in\mathbb{N}_{t}}f_{\left.{\bf y}_{i}\right|X_{i}}\left(\left.\boldsymbol{y}_{i}\right|x_{i}\right)\right)f_{{\cal X}_{t}}\left(\mathpzc{x}_{t}\right)\nonumber \\
 & =\left(\prod_{i\in\mathbb{N}_{t}}\dfrac{\exp\left(\dfrac{\overline{\boldsymbol{y}}_{i}^{\boldsymbol{T}}{\bf C}_{i}^{-1}\overline{\boldsymbol{y}}_{i}}{-2}\right)}{\sqrt{\det\left({\bf C}_{i}\right)\left(2\pi\right)^{N}}}\right)f_{{\cal X}_{t}}\left(\mathpzc{x}_{t}\right)\nonumber \\
 & \equiv\dfrac{\exp\left({\displaystyle \sum_{i\in\mathbb{N}_{t}}}\dfrac{\overline{\boldsymbol{y}}_{i}^{\boldsymbol{T}}{\bf C}_{i}^{-1}\overline{\boldsymbol{y}}_{i}}{-2}\right)}{\left({\displaystyle \prod_{i\in\mathbb{N}_{t}}}\sqrt{\det\left({\bf C}_{i}\right)}\right)\sqrt{\left(2\pi\right)^{N\left(t+1\right)}}}f_{{\cal X}_{t}}\left(\mathpzc{x}_{t}\right),\label{eq:f_normal}
\end{flalign}
where, for all $t\in\mathbb{N}$,
\begin{flalign}
\overline{\boldsymbol{y}}_{t}\equiv\overline{\boldsymbol{y}}_{t}\left(x_{t}\right) & \triangleq\boldsymbol{y}_{t}-\boldsymbol{\mu}_{t}\left(x_{t}\right)\in\mathbb{R}^{N\times1}\quad\text{and}\\
{\bf C}_{t}\equiv{\bf C}_{t}\left(x_{t}\right) & \equiv\boldsymbol{\Sigma}_{t}\left(x_{t}\right)+\sigma_{\xi}^{2}{\bf I}_{N\times N}\in{\cal D}_{{\bf C}},
\end{flalign}
where ${\cal D}_{{\bf C}}$ constitutes a bounded subset of $\mathbb{R}^{N\times N}$.
From (\ref{eq:f_tilde}) and (\ref{eq:f_normal}), it is obvious that
the sufficient condition (\ref{eq:ABSOLUTE_Densities}) of Theorem
1 is satisfied (actually, in this case, we have an equivalence; as
a result, the change of measure is an invertible transformation).
Applying Theorem 1, (\ref{eq:Changed_Expectation}) must be true by
defining the $\left\{ \mathscr{H}_{t}\right\} $-adapted stochastic
process
\begin{align}
\Lambda_{t}\equiv\Lambda_{t}\left({\cal X}_{t},\boldsymbol{{\cal Y}}_{t}\right) & \triangleq\dfrac{f_{t}\left({\cal X}_{t},\boldsymbol{{\cal Y}}_{t}\right)}{\widetilde{f}_{t}\left({\cal X}_{t},\boldsymbol{{\cal Y}}_{t}\right)}\equiv\dfrac{f_{\left.\boldsymbol{{\cal Y}}_{t}\right|{\cal X}_{t}}\left(\left.\boldsymbol{{\cal Y}}_{t}\right|{\cal X}_{t}\right)}{\widetilde{f}_{\left.\boldsymbol{{\cal Y}}_{t}\right|{\cal X}_{t}}\left(\left.\boldsymbol{{\cal Y}}_{t}\right|{\cal X}_{t}\right)}\equiv\dfrac{f_{\left.\boldsymbol{{\cal Y}}_{t}\right|{\cal X}_{t}}\left(\left.\boldsymbol{{\cal Y}}_{t}\right|{\cal X}_{t}\right)}{\widetilde{f}_{\boldsymbol{{\cal Y}}_{t}}\left(\boldsymbol{{\cal Y}}_{t}\right)}\nonumber \\
 & \equiv\dfrac{\exp\left({\displaystyle \sum_{i\in\mathbb{N}_{t}}}\dfrac{\left\Vert {\bf y}_{i}\right\Vert _{2}^{2}-\left(\overline{{\bf y}}_{i}\left(X_{i}\right)\right)^{\boldsymbol{T}}\left({\bf C}_{i}\left(X_{i}\right)\right)^{-1}\overline{{\bf y}}_{i}\left(X_{i}\right)}{2}\right)}{{\displaystyle \prod_{i\in\mathbb{N}_{t}}}\sqrt{\det\left({\bf C}_{i}\left(X_{i}\right)\right)}},
\end{align}
or, alternatively,
\begin{equation}
\Lambda_{t}\equiv{\displaystyle \prod_{i\in\mathbb{N}_{t}}}\lambda_{i}\triangleq{\displaystyle \prod_{i\in\mathbb{N}_{t}}}\dfrac{\exp\left({\displaystyle \sum_{i\in\mathbb{N}_{t}}}\dfrac{\left\Vert {\bf y}_{i}\right\Vert _{2}^{2}-\left(\overline{{\bf y}}_{i}\left(X_{i}\right)\right)^{\boldsymbol{T}}\left({\bf C}_{i}\left(X_{i}\right)\right)^{-1}\overline{{\bf y}}_{i}\left(X_{i}\right)}{2}\right)}{\sqrt{\det\left({\bf C}_{i}\left(X_{i}\right)\right)}},
\end{equation}
therefore completing the proof.
\end{proof}

\subsection{Weak \& ${\cal C}$-Weak Convergence of (Random) Probability Measures}

In the analysis that will take place in Section IV, we will make use
of the notions of weak and conditionally weak (${\cal C}$-weak) convergence
of sequences of probability measures. Thus, let us define these notions
of stochastic convergence consistently, suited at least for the purposes
of our investigation.
\begin{defn}
\label{Def_WEAK}\textbf{(Weak Convergence \cite{BillingsleyMeasures})}
Let ${\cal S}$ be an arbitrary metric space, let $\mathscr{S}\triangleq\mathscr{B}\left({\cal S}\right)$
be the associated Borel $\sigma$-algebra and consider a sequence
of probability measures $\left\{ \pi_{n}\right\} _{n\in\mathbb{N}}$
on $\mathscr{S}$. If $\pi$ constitutes another ``limit'' probability
measure on $\mathscr{S}$ such that \renewcommand{\arraystretch}{1.2}
\begin{equation}
\begin{array}{c}
{\displaystyle \lim_{n\rightarrow\infty}}\pi_{n}\left({\cal A}\right)=\pi\left({\cal A}\right),\\
\forall{\cal A}\in\mathscr{S}\text{ such that }\pi\left(\partial{\cal A}\right)\equiv0,
\end{array}
\end{equation}
\renewcommand{\arraystretch}{1}where $\partial{\cal A}$ denotes
the boundary set of the Borel set ${\cal A}$, then we say that the
sequence $\left\{ \pi_{n}\right\} _{n\in\mathbb{N}}$ converges to
$\pi$ \textit{weakly} or \textit{in the weak sense} and we equivalently
write
\begin{equation}
\pi_{n}\stackrel[n\rightarrow\infty]{{\cal W}}{\longrightarrow}{\cal \pi}.
\end{equation}

\end{defn}
\vspace{-0.5cm}

\noindent \begin{center}
\rule[0.5ex]{0.5\columnwidth}{0.5pt}
\par\end{center}

\vspace{-0.3cm}

Of course, weak convergence of probability measures is equivalent
to weak convergence or convergence in distribution, in case we are
given sequences of $\left(S,\mathscr{S}\right)$-valued random variables
whose induced probability measures converge in the aforementioned
sense.

Next, we present a definition for conditionally weak convergence of
probability measures. To avoid possibly complicating technicalities,
this definition is not presented in full generality. Rather, it is
presented in an appropriately specialized form, which will be used
later on, in the analysis that follows.
\begin{defn}
\label{Def_COND_WEAK}\textbf{(Conditionally Weak Convergence)} Let
$\left(\Omega,\mathscr{F},{\cal P}\right)$ be a base probability
triplet and consider the measurable spaces $\left({\cal S}_{i},\mathscr{S}_{i}\triangleq\mathscr{B}\left({\cal S}_{i}\right)\right),i=\left\{ 1,2\right\} $,
where ${\cal S}_{1}$ and ${\cal S}_{2}$ constitute a complete separable
metric (Polish) space and an arbitrary metric space, respectively.
Also, let $\left\{ X_{1}^{n}:\Omega\rightarrow{\cal S}_{1}\right\} _{n\in\mathbb{N}}$
be a sequence of random variables, let $X_{2}:\Omega\rightarrow{\cal S}_{2}$
be another random variable and consider the sequence of \textit{(regular)
induced conditional probability distributions (or measures)} ${\cal P}_{\left.X_{1}^{n}\right|X_{2}}^{n}:\mathscr{S}_{1}\times\Omega\rightarrow\left[0,1\right]$,
such that
\begin{equation}
{\cal P}_{\left.X_{1}^{n}\right|X_{2}}^{n}\left(\left.{\cal A}\right|X_{2}\left(\omega\right)\right)\equiv{\cal P}\left(\left.X_{1}^{n}\in{\cal A}\right|\sigma\left\{ X_{2}\right\} \right),
\end{equation}
${\cal P}-a.e.$, for any Borel set ${\cal A}\in\mathscr{S}_{1}$.
If $X_{1}:\Omega\rightarrow{\cal S}_{1}$ constitutes a ``limit''
random variable, whose induced conditional measure ${\cal P}_{\left.X_{1}\right|X_{2}}:\mathscr{S}_{1}\times\Omega\rightarrow\left[0,1\right]$
is such that \renewcommand{\arraystretch}{1.3}
\begin{equation}
\begin{array}{c}
{\displaystyle \lim_{n\rightarrow\infty}}{\cal P}_{\left.X_{1}^{n}\right|X_{2}}^{n}\left(\left.{\cal A}\right|X_{2}\left(\omega\right)\right)={\cal P}_{\left.X_{1}\right|X_{2}}\left(\left.{\cal A}\right|X_{2}\left(\omega\right)\right),\\
\forall{\cal A}\in\mathscr{S}_{1}\text{ such that }\pi\left(\partial{\cal A}\right)\equiv0\text{ and }{\cal P}-a.e.,
\end{array}
\end{equation}
\renewcommand{\arraystretch}{1}then we say that the sequence $\left\{ {\cal P}_{\left.X_{1}^{n}\right|X_{2}}^{n}\right\} _{n\in\mathbb{N}}$
converges to ${\cal P}_{\left.X_{1}\right|X_{2}}$ \textit{conditionally
weakly (}${\cal C}$\textit{-weakly)} or \textit{in the conditionally
weak (}${\cal C}$\textit{-weak) sense} and we equivalently write
\begin{equation}
{\cal P}_{\left.X_{1}^{n}\right|X_{2}}^{n}\left(\left.\cdot\right|X_{2}\right)\stackrel[n\rightarrow\infty]{{\cal W}}{\longrightarrow}{\cal P}_{\left.X_{1}\right|X_{2}}\left(\left.\cdot\right|X_{2}\right).
\end{equation}

\end{defn}
\vspace{-0.3cm}

\noindent \begin{center}
\rule[0.5ex]{0.5\columnwidth}{0.5pt}
\par\end{center}

\vspace{-0.3cm}

\begin{rem}
Actually, ${\cal C}$-weak convergence, as defined above, is strongly
related to the more general concepts of \textit{almost sure weak convergence}
and random probability measures. For instance, the reader is referred
to the related articles \cite{Berti2006} and \cite{Grubel2014}.\hfill{}\ensuremath{\blacksquare}
\end{rem}
Further, the following lemma characterizes weak convergence of probability
measures (and random variables) \cite{BillingsleyMeasures}.
\begin{lem}
\noindent \label{Portmanteau-1}\textbf{\textup{(Weak Convergence
\& Expectations)}} Let ${\cal S}$ be an arbitrary metric space and
let $\mathscr{S}\triangleq\mathscr{B}\left({\cal S}\right)$. Suppose
we are given a sequence of random variables $\left\{ X^{n}\right\} _{n\in\mathbb{N}}$
and a ``limit'' $X$, all $\left(S,\mathscr{S}\right)$-valued,
but possibly defined on different base probability spaces, with $\left\{ {\cal P}_{X^{n}}\right\} _{n\in\mathbb{N}}$
and ${\cal P}_{X}$ being their induced probability measures on $\mathscr{S}$,
respectively. Then,
\begin{equation}
X^{n}\stackrel[n\rightarrow\infty]{{\cal D}}{\longrightarrow}X\Leftrightarrow{\cal P}_{X^{n}}\stackrel[n\rightarrow\infty]{{\cal W}}{\longrightarrow}{\cal P}_{X},
\end{equation}
if and only if
\begin{equation}
\mathbb{E}\left\{ f\left(X^{n}\right)\right\} \equiv{\displaystyle \int_{{\cal S}}}f\mathrm{d}{\cal P}_{X^{n}}{\displaystyle \underset{n\rightarrow\infty}{\longrightarrow}}{\displaystyle \int_{{\cal S}}}f\mathrm{d}{\cal P}_{X}\equiv\mathbb{E}\left\{ f\left(X\right)\right\} ,\label{eq:CONV_EXP}
\end{equation}
for all bounded, continuous functions $f:{\cal S}\rightarrow\mathbb{R}$.
\end{lem}

Of course, if we replace weak convergence by ${\cal C}$-weak convergence,
Lemma \ref{Portmanteau-1} continues to hold, but, in this case, (\ref{eq:CONV_EXP})
should be understood in the almost everywhere sense (see, for example,
\cite{Grubel2014}). More specifically, under the generic notation
of Definition \ref{Def_COND_WEAK} and under the appropriate assumptions
according to Lemma \ref{Portmanteau-1}, it will be true that
\begin{equation}
\mathbb{E}\left\{ \left.f\left(X_{1}^{n}\right)\right|X_{2}\right\} \left(\omega\right)\underset{n\rightarrow\infty}{\longrightarrow}\mathbb{E}\left\{ \left.f\left(X_{1}\right)\right|X_{2}\right\} \left(\omega\right),
\end{equation}
for almost all $\omega\in\Omega$.

\section{Problem Formulation \& Statement of Main Results}

In this section, we formulate the problem of interest, that is, in
a nutshell, the problem of approximating a nonlinear MMSE filter by
another (asymptotically optimal) filtering operator, defined by replacing
the true process we would like to filter by an appropriate approximation.
Although we do not deal with such a problem here, such an approximation
would be chosen in order to yield a practically realizable approximate
filtering scheme. We also present the main result of the paper, establishing
sufficient conditions for convergence of the respective approximate
filters, in an indeed strong sense.

Let us start from the beginning. From Theorem \ref{thm:COMCOM}, we
know that
\begin{equation}
\mathbb{E}_{{\cal P}}\left\{ \left.X_{t}\right|\mathscr{Y}_{t}\right\} \equiv\dfrac{\mathbb{E}_{\widetilde{{\cal P}}}\left\{ \left.X_{t}\Lambda_{t}\right|\mathscr{Y}_{t}\right\} }{\mathbb{E}_{\widetilde{{\cal P}}}\left\{ \left.\Lambda_{t}\right|\mathscr{Y}_{t}\right\} },\quad\forall t\in\mathbb{N},\label{eq:CoM_Original}
\end{equation}
where the RHS constitutes an alternative representation for the filter
on the LHS, which constitutes the optimal in the MMSE sense estimator
of the partially observed process $X_{t}$, given the available observations
up to time $t$. If the numerical evaluation of either of the sides
of (\ref{eq:CoM_Original}) is difficult (either we are interested
in a recursive realization of the filter or not), one could focus
on the RHS, where the state and the observations constitute independent
processes, and, keeping the same observations, replace $X_{t}$ by
another process $X_{t}^{\mathtt{A}}$, called the \textit{approximation},
with \textit{resolution }or\textit{ approximation parameter} $\mathtt{A}\in\mathbb{N}$
(for simplicity), also independent of the observations (with respect
to ${\cal \widetilde{P}}$), for which the evaluation of the resulting
``filter'' might be easier. Under some appropriate, well defined
sense, the approximation to the original process improves as $\mathtt{A}\rightarrow\infty$.
This general idea of replacing the true state process with an approximation
is employed in, for instance, \cite{Kushner2001_BOOK,Kushner2008},
and will be employed here, too.

At this point, a natural question arises: Why are we complicating
things with change of measure arguments and not using $X_{t}^{\mathtt{A}}$
directly in the LHS of (\ref{eq:CoM_Original})? Indeed, using classical
results such as the Dominated Convergence Theorem, one could prove
at least pointwise convergence of the respective filter approximations.
The main and most important issue with such an approach is that, in
order for such a filter to be realizable in any way, special attention
must be paid to the choice of the approximation, regarding its stochastic
dependence on the observations process. This is due to the original
stochastic coupling between the state and the observations of the
hidden system of interest. However, using change of measure, one can
find an alternative representation of the filter process, where, under
another probability measure, the state and observations are stochastically
decoupled (independent). This makes the problem much easier, because
the approximation can also be chosen to be independent of the observations.
If we especially restrict our attention to recursive nonlinear filters,
change of measure provides a rather versatile means for discovering
recursive filter realizations. See, for example, the detailed treatment
presented in \cite{Elliott1994Hidden}.

Thus, concentrating on the RHS of (\ref{eq:CoM_Original}), we can
define an \textit{approximate filtering operator} of the process $X_{t}$,
also with resolution $\mathtt{A}\in\mathbb{N}$, as\textit{ }
\begin{equation}
{\cal E}^{\mathtt{A}}\left(\left.X_{t}\right|\mathscr{Y}_{t}\right)\triangleq\dfrac{\mathbb{E}_{\widetilde{{\cal P}}}\left\{ \left.X_{t}^{\mathtt{A}}\Lambda_{t}^{\mathtt{A}}\right|\mathscr{Y}_{t}\right\} }{\mathbb{E}_{\widetilde{{\cal P}}}\left\{ \left.\Lambda_{t}^{\mathtt{A}}\right|\mathscr{Y}_{t}\right\} },\quad\forall t\in\mathbb{N}.\label{eq:CoM_Approx}
\end{equation}
Observe that the above quantity \textit{is not a conditional expectation
of $X_{t}^{\mathtt{A}}$}, because $X_{t}^{\mathtt{A}}$ does not
follow the probability law of the true process of interest, $X_{t}$
\cite{Kushner2008}. Of course, the question is if and under which
sense,
\begin{equation}
{\cal E}^{\mathtt{A}}\left(\left.X_{t}\right|\mathscr{Y}_{t}\right)\overset{?}{\underset{\mathtt{A}\rightarrow\infty}{\longrightarrow}}\mathbb{E}_{{\cal P}}\left\{ \left.X_{t}\right|\mathscr{Y}_{t}\right\} ,
\end{equation}
that is, if and in which sense our chosen approximate filtering operator
is \textit{asymptotically optimal}, as the resolution of the approximation
increases. In other words, we are looking for a class of approximations,
whose members approximate the process $X_{t}$ well, in the sense
that the resulting approximate filtering operators converge to the
true filter as the resolution parameter increases, that is, as $\mathtt{A}\rightarrow\infty$,
and under some appropriate notion of convergence. In this respect,
below we formulate and prove the following theorem, which constitutes
the main result of this paper (recall the definition of ${\cal C}$-weak
convergence given in Section II.C). In the following, $\mathds{1}_{{\cal A}}:\mathbb{R}\rightarrow\left\{ 0,1\right\} $
denotes the indicator of the set ${\cal A}$. Also, for any Borel
set ${\cal A}$, $\mathds{1}_{{\cal A}}\left(\cdot\right)$ constitutes
a Dirac (atomic) probability measure. Equivalently, we write $\mathds{1}_{{\cal A}}\left(\cdot\right)\equiv\delta_{\left(\cdot\right)}\left({\cal A}\right)$. 
\begin{thm}
\label{CONVERGENCE_THEOREM}\textbf{\textup{(Convergence to the Optimal
Filter)}} Pick any natural $T<\infty$ and suppose either of the following:
\begin{itemize}
\item \noindent For all $t\in\mathbb{N}_{T}$, the sequence $\left\{ X_{t}^{\mathtt{A}}\right\} _{\mathtt{A}\in\mathbb{N}}$
is marginally ${\cal C}$-weakly convergent to $X_{t}$, given $X_{t}$,
that is,
\begin{equation}
{\cal P}_{\left.X_{t}^{\mathtt{A}}\right|X_{t}}^{\mathtt{A}}\left(\left.\cdot\right|X_{t}\right)\stackrel[\mathtt{A}\rightarrow\infty]{{\cal W}}{\longrightarrow}\delta_{X_{t}}\left(\cdot\right),\quad\forall t\in\mathbb{N}_{T}.\label{eq:C_Weak_Theorem}
\end{equation}

\item \noindent For all $t\in\mathbb{N}_{T}$, the sequence $\left\{ X_{t}^{\mathtt{A}}\right\} _{\mathtt{A}\in\mathbb{N}}$
is (marginally) convergent to $X_{t}$ in probability, that is,
\begin{equation}
X_{t}^{\mathtt{A}}\stackrel[\mathtt{A}\rightarrow\infty]{{\cal P}}{\longrightarrow}X_{t},\quad\forall t\in\mathbb{N}_{T}.
\end{equation}

\end{itemize}
Then, there exists a measurable subset $\widehat{\Omega}_{T}\subseteq\Omega$
with ${\cal P}$-measure at least $1-\left(T+1\right)^{1-CN}\exp\left(-CN\right)$,
such that
\begin{equation}
\lim_{\mathtt{A}\rightarrow\infty}\sup_{t\in\mathbb{N}_{T}}\sup_{\omega\in\widehat{\Omega}_{T}}\left|{\cal E}^{\mathtt{A}}\left(\left.X_{t}\right|\mathscr{Y}_{t}\right)-\mathbb{E}_{{\cal P}}\left\{ \left.X_{t}\right|\mathscr{Y}_{t}\right\} \right|\left(\omega\right)\equiv0,
\end{equation}
for any free, finite constant $C\ge1$. In other words, the convergence
of the respective approximate filtering operators is compact in $t\in\mathbb{N}$
and, with probability at least $1-\left(T+1\right)^{1-CN}\exp\left(-CN\right)$,
uniform in $\omega$.
\end{thm}
\noindent Interestingly, as noted in the beginning of this section,
the mode of convergence of the resulting approximate filtering operator
is particularly strong. In fact, it is interesting that, for fixed
$T$, the approximate filter ${\cal E}^{\mathtt{A}}\left(\left.X_{t}\right|\mathscr{Y}_{t}\right)$
converges to $\mathbb{E}_{{\cal P}}\left\{ \left.X_{t}\right|\mathscr{Y}_{t}\right\} $
(uniformly) in a set that approaches the certain event, exponentially
in $N$. That is, convergence to the optimal filter \textbf{\textit{tends}}\textit{
to be in the uniformly almost everywhere sense, at an exponential
rate (in $N$)}. Consequently, it is revealed that the dimensionality
of the observations process essentially stabilizes the behavior of
the approximate filter, in a stochastic sense. Along the lines of
the discussion presented above, it is clear that Theorem \ref{CONVERGENCE_THEOREM}
provides a way of quantitatively justifying Egoroff's theorem \cite{Richardson2009measure},
which bridges almost uniform convergence with almost sure convergence,
however in an indeed abstract fashion. 
\begin{rem}
The ${\cal C}$-weak convergence condition (\ref{eq:C_Weak_Theorem})
is a rather strong one. In particular, as we show later in Lemma \ref{L1WEAK_Lemma}
(see Section IV), it implies ${\cal L}_{1}$ convergence, which means
that it also implies (marginal) convergence in probability (which
constitutes the alternative sufficient condition of Theorem \ref{CONVERGENCE_THEOREM}).
In simple words, (\ref{eq:C_Weak_Theorem}) resembles a situation
where, at any time step, one is given or defines an approximation
to the original process, in the sense that, \textit{conditioned on
the original process} at the same time step, the probability of being
equal to the latter approaches unity. At this point, because ${\cal C}$-weak
convergence is stronger than (and implies) convergence in probability,
one could wonder why we presented both as alternative sufficient conditions
for filter convergence in Theorem \ref{CONVERGENCE_THEOREM} (and
also in Lemma \ref{RD_Derivatives_Lemma} presented in Section IV).
The reason is that, contrary to convergence in probability, condition
(\ref{eq:C_Weak_Theorem}) provides a nice \textit{structural} criterion
for \textit{constructing} state process approximations in a natural
way, which is also consistent with our intuition: If, at any time
step, we could observe the value of true state process, then the respective
value of the approximation at that same time step should be ``sufficiently
close'' to the value of the state. Condition (\ref{eq:C_Weak_Theorem})
expresses this intuitive idea and provides a version of the required
sense of ``closeness''.\hfill{}\ensuremath{\blacksquare}
\end{rem}
In order to demonstrate the applicability of Theorem \ref{CONVERGENCE_THEOREM},
as well as demystify the ${\cal C}$-weak convergence condition (\ref{eq:C_Weak_Theorem}),
let us present a simple but illustrative example. The example refers
to a class of approximate grid based filters, based on the so called
marginal approximation \cite{KalPetGRID2014,Pages2005optimal}, according
to which the (compactly restricted) state process is fed into a uniform
spatial quantizer of variable resolution. As we will see, this intuitively
reasonable approximation idea constitutes a simple instance of the
condition (\ref{eq:C_weak_COND}).

More specifically, assume that $X_{t}\in\left[a,b\right]\equiv{\cal Z}$,
$\forall t\in\mathbb{N}$, almost surely. Let us discretize ${\cal Z}$
uniformly into $\mathtt{A}$ subintervals, of identical length, called
cells. The $l$-th cell and its respective center are denoted as ${\cal Z}_{\mathtt{A}}^{l}$
and $x_{\mathtt{A}}^{l},l\in\mathbb{N}_{\mathtt{A}}^{+}$. Then, letting
${\cal X}_{\mathtt{A}}\triangleq\left\{ x_{\mathtt{A}}^{l}\right\} _{l\in\mathbb{N}_{L_{S}}^{+}}$,
the \textit{quantizer} ${\cal Q}_{\mathtt{A}}:\left({\cal Z},\mathscr{B}\left({\cal Z}\right)\right)\mapsto\left({\cal X}_{\mathtt{A}},2^{{\cal X}_{\mathtt{A}}}\right)$
is defined as the bijective and measurable function which uniquely
maps the $l$-th cell to the respective \textit{reconstruction point}
$x_{\mathtt{A}}^{l}$, $\forall l\in\mathbb{N}_{\mathtt{A}}^{+}$.
That is, ${\cal Q}_{\mathtt{A}}\left(x\right)\triangleq x_{\mathtt{A}}^{l}$
if and only if $x\in{\cal Z}_{\mathtt{A}}^{l}$ \cite{KalPetGRID2014}.
Having defined the quantizer ${\cal Q}_{\mathtt{A}}\left(\cdot\right)$,
the \textit{Marginal Quantization} of the state is defined as \cite{Pages2005optimal}
\begin{equation}
X_{t}^{\mathtt{A}}\left(\omega\right)\triangleq{\cal Q}_{\mathtt{A}}\left(X_{t}\left(\omega\right)\right)\in{\cal X}_{\mathtt{A}},\;\forall t\in\mathbb{N},\;{\cal P}-a.s.,
\end{equation}
where $\mathtt{A}\in\mathbb{N}$ is identified as the approximation
parameter. That is, $X_{t}$ is approximated by its nearest neighbor
on the cell grid. That is, the state is represented by a discrete
set of reconstruction points, each one of them uniquely corresponding
to a member of a partition of ${\cal Z}$.

By construction of marginal state approximations, it can be easily
shown that \cite{KalPetGRID2014}
\begin{equation}
X_{t}^{\mathtt{A}}\left(\omega\right)\stackrel[\mathtt{A}\rightarrow\infty]{{\cal P}-a.s.}{\longrightarrow}X_{t}\left(\omega\right),
\end{equation}
a fact that will be used in the following. Of course, almost sure
convergence implies convergence in probability and, as we will see,
${\cal C}$-weak convergence as well. First, let us determine the
conditional probability measure ${\cal P}_{\left.X_{t}^{\mathtt{A}}\right|X_{t}}^{\mathtt{A}}\left(\left.\mathrm{d}x\right|X_{t}\right)$.
Since knowing $X_{t}$ uniquely determines the value of $X_{t}^{\mathtt{A}}$,
it must be true that
\begin{flalign}
{\cal P}_{\left.X_{t}^{\mathtt{A}}\right|X_{t}}^{\mathtt{A}}\left(\left.\mathrm{d}x\right|X_{t}\right) & \equiv{\cal P}_{\left.{\cal Q}_{\mathtt{A}}\left(X_{t}\right)\right|X_{t}}^{\mathtt{A}}\left(\left.\mathrm{d}x\right|X_{t}\right)\nonumber \\
 & \equiv\delta_{{\cal Q}_{\mathtt{A}}\left(X_{t}\right)}\left(\mathrm{d}x\right),\quad{\cal P}-a.s..
\end{flalign}
However, from Lemma \ref{Portmanteau-1}, we know that weak convergence
of measures is equivalent to showing that the expectations $\mathbb{E}\left\{ \left.f\left(X_{t}^{\mathtt{A}}\right)\right|X_{t}\right\} $
converge to $\mathbb{E}\left\{ \left.f\left(X_{t}\right)\right|X_{t}\right\} \equiv f\left(X_{t}\right)$,
for all bounded and continuous $f\left(\cdot\right)$, almost everywhere.
Indeed,
\begin{flalign}
\mathbb{E}\left\{ \left.f\left(X_{t}^{\mathtt{A}}\right)\right|X_{t}\right\} \left(\omega\right) & \equiv\int_{{\cal Z}}f\left(x\right){\cal P}_{\left.X_{t}^{\mathtt{A}}\right|X_{t}}^{\mathtt{A}}\left(\left.\mathrm{d}x\right|X_{t}\left(\omega\right)\right)\nonumber \\
 & \equiv\int_{{\cal Z}}f\left(x\right)\delta_{{\cal Q}_{\mathtt{A}}\left(X_{t}\left(\omega\right)\right)}\left(\mathrm{d}x\right)\nonumber \\
 & \equiv f\left({\cal Q}_{\mathtt{A}}\left(X_{t}\left(\omega\right)\right)\right)\stackrel[\mathtt{A}\rightarrow\infty]{{\cal P}-a.s.}{\longrightarrow}f\left(X_{t}\left(\omega\right)\right),
\end{flalign}
due to the continuity of $f\left(\cdot\right)$. Consequently, we
have shown that
\begin{equation}
{\cal P}_{\left.X_{t}^{\mathtt{A}}\right|X_{t}}^{\mathtt{A}}\left(\left.\cdot\right|X_{t}\right)\equiv\delta_{{\cal Q}_{\mathtt{A}}\left(X_{t}\right)}\left(\cdot\right)\stackrel[\mathtt{A}\rightarrow\infty]{{\cal W}}{\longrightarrow}\delta_{X_{t}}\left(\cdot\right),
\end{equation}
fulfilling the first requirement of Theorem \ref{CONVERGENCE_THEOREM}.
This very simple example constitutes the basis for constructing more
complicated and cleverly designed state approximations (for example,
using stochastic quantizers). The challenge here is to come up with
such approximations exhibiting nice properties, which would potentially
lead to the development of effective approximate recursive or, in
general, sequential filtering schemes, well suited for dynamic inference
in complex partially observable stochastic nonlinear systems. As far
as grid based approximate recursive filtering is concerned, a relatively
complete discussion of the problem is presented in the recent paper
\cite{KalPetGRID2014}, where marginal state approximations are also
treated in full generality.

An important and direct consequence of Theorem \ref{CONVERGENCE_THEOREM},
also highlighted by the example presented above, is that, interestingly,
the nature of the state process is completely irrelevant when one
is interested in convergence of the respective approximate filters,
in the respective sense of the aforementioned theorem. This fact has
the following pleasing and intuitive interpretation: It implies that
if any of the two conditions of Theorem \ref{CONVERGENCE_THEOREM}
are satisfied, then we should forget about the internal stochastic
structure of the state, and instead focus exclusively on the way the
latter is being observed through time. That is, \textit{we do not
really care about what we partially observe, but how well we observe
it; and if we observe it well, we can filter it well, too}. Essentially,
the observations should constitute a stable functional of the state,
of course in some well defined sense. In this work, this notion of
stability is expressed precisely through Assumption 1 and 2, presented
earlier in Section II.

Note, however, that the existence of a consistent approximate filter
in the sense of Theorem \ref{CONVERGENCE_THEOREM} does not automatically
imply that this filter will be efficiently implementable; usually,
we would like such a filter to admit a recursive/sequential representation
(or possibly a semirecursive one \cite{Daum2005}). As it turns out,
this can happen when the chosen state approximation admits a valid
semimartingale type representation (in addition to satisfying one
of the sufficient conditions of Theorem \ref{CONVERGENCE_THEOREM}).
For example, the case where the state is Markovian and the chosen
state approximation is of the marginal type, discussed in the basic
example presented above, is treated in detail in \cite{KalPetGRID2014}.
\begin{rem}
The filter representation (\ref{eq:CoM_Original}) coincides with
the respective expression employed in importance sampling \cite{PARTICLE2002tutorial,RebeschiniRamon2013_1}.
Since, under the alternative measure $\widetilde{{\cal P}}$, the
observations and state constitute statistically independent processes,
one can directly sample from the (joint) distribution of the state,
fixing the observations to their respective value at each time $t$
(of course, assuming that a relevant ``sampling device'' exists).
However, note that that due to the assumptions of Theorem \ref{CONVERGENCE_THEOREM},
related at least to convergence in probability of the corresponding
state approximations, the aforementioned result cannot be used directly
in order to show convergence of importance sampling or related particle
filtering techniques, which are directly related to empirical measures.
The possible ways Theorem \ref{CONVERGENCE_THEOREM} can be utilized
in order to provide asymptotic guarantees for particle filtering (using
additional assumptions) constitutes an interesting open topic for
further research.\hfill{}\ensuremath{\blacksquare}
\end{rem}
The rest of the paper is fully devoted in the detailed proof of Theorem
\ref{CONVERGENCE_THEOREM}.

\section{Proof of Theorem \ref{CONVERGENCE_THEOREM}}

In order to facilitate the presentation, the proof is divided in a
number of subsections.

\subsection{Two Basic Lemmata, Linear Algebra - Oriented}

Parts of the following useful results will be employed several times
in the analysis that follows\footnote{In this paper, Lemma \ref{Elementary_Lemma_2} presented in this subsection
will be applied only for scalars (and where the metric considered
coincides with the absolute value). However, the general version of
the result (considering matrices and submultiplicative norms) is presented
for the sake of generality.}. 
\begin{lem}
\label{Elementary_Lemma_1}Consider arbitrary matrices ${\bf A}\in\mathbb{C}^{N_{1}\times M_{1}}$,
${\bf B}\in\mathbb{C}^{N_{1}\times M_{1}}$, ${\bf X}\in\mathbb{C}^{M_{2}\times N_{2}}$,
${\bf Y}\in\mathbb{C}^{M_{2}\times N_{2}}$, and let $\left\Vert \cdot\right\Vert _{\mathfrak{M}}$
be any matrix norm. Then, the following hold:
\begin{itemize}
\item If either 

\begin{itemize}
\item $N_{1}\equiv M_{1}\equiv1$, or
\item $N_{1}\equiv N_{2}\equiv M_{1}\equiv M_{2}$ and $\left\Vert \cdot\right\Vert _{\mathfrak{M}}$
is submultiplicative,
\end{itemize}

then
\begin{align}
\left\Vert {\bf A}{\bf X}-{\bf B}{\bf Y}\right\Vert _{\mathfrak{M}} & \le\left\Vert {\bf A}\right\Vert _{\mathfrak{M}}\left\Vert {\bf X}-{\bf Y}\right\Vert _{\mathfrak{M}}+\left\Vert {\bf Y}\right\Vert _{\mathfrak{M}}\left\Vert {\bf A}-{\bf B}\right\Vert _{\mathfrak{M}}.\label{eq:FIRST}
\end{align}

\item If $N_{2}\equiv1$, $M_{1}\equiv M_{2}$ and $\left\Vert \cdot\right\Vert _{\mathfrak{M}}$
constitutes any subordinate matrix norm to the $\ell_{p}$ vector
norm, $\left\Vert \cdot\right\Vert _{p}$, then
\begin{align}
\left\Vert {\bf A}{\bf X}-{\bf B}{\bf Y}\right\Vert _{p} & \le\left\Vert {\bf A}\right\Vert _{\mathfrak{M}}\left\Vert {\bf X}-{\bf Y}\right\Vert _{p}+\left\Vert {\bf Y}\right\Vert _{p}\left\Vert {\bf A}-{\bf B}\right\Vert _{\mathfrak{M}}.\label{eq:FIRST-1}
\end{align}

\end{itemize}
\end{lem}
\begin{proof}[Proof of Lemma \ref{Elementary_Lemma_1}]
We prove the result only for the case where $N_{1}\equiv N_{2}\equiv M_{1}\equiv M_{2}$
and $\left\Vert \cdot\right\Vert _{\mathfrak{M}}$ is submultiplicative.
By definition of such a matrix norm,
\begin{flalign}
\left\Vert {\bf A}{\bf X}-{\bf B}{\bf Y}\right\Vert _{\mathfrak{M}} & \equiv\left\Vert {\bf A}{\bf X}+{\bf A}{\bf Y}-{\bf A}{\bf Y}-{\bf B}{\bf Y}\right\Vert _{\mathfrak{M}}\nonumber \\
 & \equiv\left\Vert {\bf A}\left({\bf X}-{\bf Y}\right)+\left({\bf A}-{\bf B}\right){\bf Y}\right\Vert _{\mathfrak{M}}\nonumber \\
 & \le\left\Vert {\bf A}\left({\bf X}-{\bf Y}\right)\right\Vert _{\mathfrak{M}}+\left\Vert \left({\bf A}-{\bf B}\right){\bf Y}\right\Vert _{\mathfrak{M}}\nonumber \\
 & \le\left\Vert {\bf A}\right\Vert _{\mathfrak{M}}\left\Vert \left({\bf X}-{\bf Y}\right)\right\Vert _{\mathfrak{M}}+\left\Vert {\bf Y}\right\Vert _{\mathfrak{M}}\left\Vert \left({\bf A}-{\bf B}\right)\right\Vert _{\mathfrak{M}},
\end{flalign}
apparently completing the proof. The results for the other two cases
considered in Lemma \ref{Elementary_Lemma_1} can be readily shown
following similar procedure. \end{proof}
\begin{lem}
\label{Elementary_Lemma_2}Consider the collections of arbitrary,
square matrices
\begin{flalign*}
\left\{ {\bf A}_{i}\in\mathbb{C}^{N\times N}\right\} _{i\in\mathbb{N}_{n}}\quad\text{and}\quad & \left\{ {\bf B}_{i}\in\mathbb{C}^{N\times N}\right\} _{i\in\mathbb{N}_{n}}.
\end{flalign*}
 Then, for any submultiplicative matrix norm $\left\Vert \cdot\right\Vert _{\mathfrak{M}}$,
it is true that
\begin{align}
\left\Vert \prod_{i=0}^{n}{\bf A}_{i}-\prod_{i=0}^{n}{\bf B}_{i}\right\Vert _{\mathfrak{M}} & \le\sum_{i=0}^{n}\left(\prod_{j=0}^{i-1}\left\Vert {\bf A}_{j}\right\Vert _{\mathfrak{M}}\right)\left(\prod_{j=i+1}^{n}\left\Vert {\bf B}_{j}\right\Vert _{\mathfrak{M}}\right)\left\Vert {\bf A}_{i}-{\bf B}_{i}\right\Vert _{\mathfrak{M}}.\label{eq:Product_Inequality}
\end{align}
\end{lem}
\begin{proof}[Proof of Lemma \ref{Elementary_Lemma_2}]
Applying Lemma \ref{Elementary_Lemma_1} to the LHS of (\ref{eq:Product_Inequality}),
we get
\begin{flalign}
\left\Vert \prod_{i=0}^{n}{\bf A}_{i}-\prod_{i=0}^{n}{\bf B}_{i}\right\Vert _{\mathfrak{M}} & \equiv\left\Vert {\bf A}_{0}\prod_{i=1}^{n}{\bf A}_{i}-{\bf B}_{0}\prod_{i=1}^{n}{\bf B}_{i}\right\Vert _{\mathfrak{M}}\nonumber \\
 & \le\left\Vert {\bf A}_{0}\right\Vert _{\mathfrak{M}}\left\Vert \prod_{i=1}^{n}{\bf A}_{i}-\prod_{i=1}^{n}{\bf B}_{i}\right\Vert _{\mathfrak{M}}+\left\Vert \prod_{i=1}^{n}{\bf B}_{i}\right\Vert _{\mathfrak{M}}\left\Vert {\bf A}_{0}-{\bf B}_{0}\right\Vert _{\mathfrak{M}}.
\end{flalign}
The repeated application of Lemma \ref{Elementary_Lemma_1} to the
quantity multiplying $\left\Vert {\bf A}_{0}\right\Vert _{\mathfrak{M}}$
on the RHS of the expression above yields
\begin{multline}
\left\Vert \prod_{i=0}^{n}{\bf A}_{i}-\prod_{i=0}^{n}{\bf B}_{i}\right\Vert _{\mathfrak{M}}\le\left\Vert {\bf A}_{0}\right\Vert _{\mathfrak{M}}\left\Vert {\bf A}_{1}\right\Vert _{\mathfrak{M}}\left\Vert \prod_{i=2}^{n}{\bf A}_{i}-\prod_{i=2}^{n}{\bf B}_{i}\right\Vert _{\mathfrak{M}}\\
+\left\Vert {\bf A}_{0}\right\Vert _{\mathfrak{M}}\left\Vert \prod_{i=2}^{n}{\bf B}_{i}\right\Vert _{\mathfrak{M}}\left\Vert {\bf A}_{1}-{\bf B}_{1}\right\Vert _{\mathfrak{M}}+\left\Vert \prod_{i=1}^{n}{\bf B}_{i}\right\Vert _{\mathfrak{M}}\left\Vert {\bf A}_{0}-{\bf B}_{0}\right\Vert _{\mathfrak{M}},\label{eq:Prod_Lemma}
\end{multline}
where, the ``temporal pattern'' is apparent. Indeed, iterating (\ref{eq:Prod_Lemma})
and proceeding inductively, we end up with the bound
\begin{align}
\left\Vert \prod_{i=0}^{n}{\bf A}_{i}-\prod_{i=0}^{n}{\bf B}_{i}\right\Vert _{\mathfrak{M}} & \le\sum_{i=0}^{n}\left(\prod_{j=0}^{i-1}\left\Vert {\bf A}_{j}\right\Vert _{\mathfrak{M}}\right)\left\Vert \prod_{j=i+1}^{n}{\bf B}_{j}\right\Vert _{\mathfrak{M}}\left\Vert {\bf A}_{i}-{\bf B}_{i}\right\Vert _{\mathfrak{M}}
\end{align}
and the result readily follows invoking the submultiplicativeness
of $\left\Vert \cdot\right\Vert _{\mathfrak{M}}$.
\end{proof}

\subsection{Preliminary Results}

Here, we present and prove a number of preliminary results, which
will help us towards the proof of an important lemma, which will be
the key to showing the validity of Theorem \ref{CONVERGENCE_THEOREM}.

First, under Assumption 2, stated in Section II.A, the following trivial
lemmata hold.
\begin{lem}
\label{Elementary_Corollary_1}Each member of the functional family
$\left\{ \boldsymbol{\Sigma}_{t}:{\cal Z}\mapsto{\cal D}_{\boldsymbol{\Sigma}}\right\} _{t\in\mathbb{N}}$
is Lipschitz continuous on ${\cal Z}$, in the Euclidean topology
induced by the Frobenius norm. That is, $\forall t\in\mathbb{N}$,
\begin{equation}
\left\Vert \boldsymbol{\Sigma}_{t}\left(x\right)-\boldsymbol{\Sigma}_{t}\left(y\right)\right\Vert _{F}\le\left(NK_{\boldsymbol{\Sigma}}\right)\left|x-y\right|,
\end{equation}
$\forall\left(x,y\right)\in{\cal Z}\times{\cal Z}$, for the same
constant $K_{\boldsymbol{\Sigma}}\in\mathbb{R}_{+}$, as defined in
Assumption 2. The same also holds for the family $\left\{ {\bf C}_{t}:{\cal Z}\mapsto{\cal D}_{{\bf C}}\right\} _{t\in\mathbb{N}}$.\end{lem}
\begin{proof}[Proof of Lemma \ref{Elementary_Corollary_1}]
By definition of the Frobenius norm,
\begin{flalign}
\left\Vert \boldsymbol{\Sigma}_{t}\left(x\right)-\boldsymbol{\Sigma}_{t}\left(y\right)\right\Vert _{F} & \equiv\sqrt{\sum_{\left(i,j\right)\in\mathbb{N}_{N}^{+}\times\mathbb{N}_{N}^{+}}\left(\boldsymbol{\Sigma}_{t}^{ij}\left(x\right)-\boldsymbol{\Sigma}_{t}^{ij}\left(y\right)\right)^{2}}\nonumber \\
 & \le\sqrt{\sum_{\left(i,j\right)\in\mathbb{N}_{N}^{+}\times\mathbb{N}_{N}^{+}}K_{\boldsymbol{\Sigma}}^{2}\left|x-y\right|^{2}}\nonumber \\
 & \equiv\sqrt{N^{2}K_{\boldsymbol{\Sigma}}^{2}\left|x-y\right|^{2}},\quad\forall t\in\mathbb{N}
\end{flalign}
and our first claim follows. The second follows trivially if we recall
the definition of each ${\bf C}_{t}\left(x\right)$. \end{proof}
\begin{lem}
\label{Elementary_Corollary_2}For each member of the functional family
$\left\{ {\bf C}_{t}:{\cal Z}\mapsto{\cal D}_{{\bf C}}\right\} _{t\in\mathbb{N}}$,
it is true that, $\forall t\in\mathbb{N}$,
\begin{align}
\left|\det\left({\bf C}_{t}\left(x\right)\right)-\det\left({\bf C}_{t}\left(y\right)\right)\right| & \le\left(NK_{DET}\right)\left\Vert {\bf C}_{t}\left(x\right)-{\bf C}_{t}\left(y\right)\right\Vert _{F},
\end{align}
$\forall\left(x,y\right)\in{\cal Z}\times{\cal Z}$, for some bounded
constant $K_{DET}\equiv K_{DET}\left(N\right)\in\mathbb{R}_{+}$,
possibly dependent on $N$ but independent of $t$.\end{lem}
\begin{proof}[Proof of Lemma \ref{Elementary_Corollary_2}]
As a consequence of the fact that the determinant of a matrix can
be expressed as a polynomial function in $N^{2}$ variables (for example,
see the Leibniz formula), it must be true that, $\forall t\in\mathbb{N}$,
\begin{flalign}
\left|\det\left({\bf C}_{t}\left(x\right)\right)-\det\left({\bf C}_{t}\left(y\right)\right)\right| & \le K_{DET}\sum_{\left(i,j\right)\in\mathbb{N}_{N}^{+}\times\mathbb{N}_{N}^{+}}\left|{\bf C}_{t}^{ij}\left(x\right)-{\bf C}_{t}^{ij}\left(y\right)\right|\nonumber \\
 & \equiv K_{DET}\left\Vert {\bf C}_{t}\left(x\right)-{\bf C}_{t}\left(y\right)\right\Vert _{1},
\end{flalign}
where the constant $K_{DET}$ depends on \textit{maximized} (using
the fact that the domain ${\cal D}_{{\bf C}}$ is bounded) $\left(N-1\right)$-fold
products of elements of ${\bf C}_{t}\left(x\right)$ and ${\bf C}_{t}\left(y\right)$,
with respect to $x$ (resp. $y$) and $t$. Consequently, although
$K_{DET}$ may depend on $N$, it certainly does not depend on $t$.
Now, since the $\ell_{1}$ entrywise norm of an $N\times N$ matrix
corresponds to the norm of a vector with $N^{2}$ elements, we may
further bound the right had side of the expression above by the Frobenius
norm of ${\bf C}_{t}\left(x\right)-{\bf C}_{t}\left(y\right)$, yielding
\begin{align}
\left|\det\left({\bf C}_{t}\left(x\right)\right)-\det\left({\bf C}_{t}\left(y\right)\right)\right| & \le NK_{DET}\left\Vert {\bf C}_{t}\left(x\right)-{\bf C}_{t}\left(y\right)\right\Vert _{F},
\end{align}
which is what we were set to prove.\end{proof}
\begin{rem}
The fact that the constant $K_{DET}$ may be a function of the dimension
of the observation vector, $N$, does not constitute a significant
problem throughout our analysis, simply because $N$ is always considered
a \textit{finite and fixed} parameter of our problem. However, it
is true that the (functional) way $N$ appears in the various constants
in our derived expressions can potentially affect speed of convergence
and, for that reason, it constitutes an important analytical aspect.
Therefore, throughout the analysis presented below, a great effort
has been made in order to keep the dependence of our bounds on $N$
within reasonable limits.\hfill{}\ensuremath{\blacksquare} 
\end{rem}
We also present another useful lemma, related to the expansiveness
of each member of the functional family $\left\{ {\bf C}_{t}^{-1}:{\cal Z}\mapsto{\cal D}_{{\bf C}^{-1}}\right\} _{t\in\mathbb{N}}$.
\begin{lem}
\label{Elementary_Lemma_3}Each member of the functional family $\left\{ {\bf C}_{t}^{-1}:{\cal Z}\mapsto{\cal D}_{{\bf C}^{-1}}\right\} _{t\in\mathbb{N}}$
is Lipschitz continuous on ${\cal Z}$, in the Euclidean topology
induced by the Frobenius norm. That is, $\forall t\in\mathbb{N}$,
\begin{equation}
\left\Vert {\bf C}_{t}^{-1}\left(x\right)-{\bf C}_{t}^{-1}\left(y\right)\right\Vert _{F}\le K_{INV}\left|x-y\right|,
\end{equation}
$\forall\left(x,y\right)\in{\cal Z}\times{\cal Z}$, for some bounded
constant $K_{INV}\equiv K_{INV}\left(N\right)\in\mathbb{R}_{+}$,
possibly dependent on $N$ but independent of $t$.\end{lem}
\begin{proof}[Proof of Lemma \ref{Elementary_Lemma_3}]
As a consequence of Laplace's formula for the determinant of a matrix
and invoking Lemma \ref{Elementary_Lemma_1}, it is true that
\begin{flalign}
 & \hspace{-2pt}\hspace{-2pt}\hspace{-2pt}\hspace{-2pt}\hspace{-2pt}\hspace{-2pt}\hspace{-2pt}\hspace{-2pt}\hspace{-2pt}\hspace{-2pt}\hspace{-2pt}\hspace{-2pt}\hspace{-2pt}\hspace{-2pt}\hspace{-2pt}\left\Vert {\bf C}_{t}^{-1}\left(x\right)-{\bf C}_{t}^{-1}\left(y\right)\right\Vert _{F}\nonumber \\
 & \equiv\left\Vert \dfrac{\mathrm{adj}\left({\bf C}_{t}\left(x\right)\right)}{\det\left({\bf C}_{t}\left(x\right)\right)}-\dfrac{\mathrm{adj}\left({\bf C}_{t}\left(y\right)\right)}{\det\left({\bf C}_{t}\left(y\right)\right)}\right\Vert _{F}\nonumber \\
 & \le\dfrac{\left\Vert \mathrm{adj}\left({\bf C}_{t}\left(x\right)\right)-\mathrm{adj}\left({\bf C}_{t}\left(y\right)\right)\right\Vert _{F}}{\det\left({\bf C}_{t}\left(x\right)\right)}+\left\Vert \mathrm{adj}\left({\bf C}_{t}\left(y\right)\right)\right\Vert _{F}\dfrac{\left|\det\left({\bf C}_{t}\left(x\right)\right)-\det\left({\bf C}_{t}\left(y\right)\right)\right|}{\det\left({\bf C}_{t}\left(x\right)\right)\det\left({\bf C}_{t}\left(y\right)\right)},\label{eq:Elem_Lemma3_1}
\end{flalign}
where $\mathrm{adj}\left({\bf A}\right)$ denotes the adjugate of
the square matrix ${\bf A}$. Since ${\bf C}_{t}\left(x\right)$ (resp.
${\bf C}_{t}\left(y\right)$) is a symmetric and positive definite
matrix, so is its adjugate. Employing one more property regarding
the eigenvalues of the adjugate \cite{SharifiADJUGATE2011} and the
fact that $\lambda_{inf}>1$, we can write
\begin{flalign}
\left\Vert \mathrm{adj}\left({\bf C}_{t}\left(y\right)\right)\right\Vert _{F} & \le\sqrt{N}\left\Vert \mathrm{adj}\left({\bf C}_{t}\left(y\right)\right)\right\Vert _{2}\nonumber \\
 & \equiv\sqrt{N}\lambda_{max}\left(\mathrm{adj}\left({\bf C}_{t}\left(y\right)\right)\right)\nonumber \\
 & \equiv\sqrt{N}\max_{i\in\mathbb{N}_{N}^{+}}\prod_{j\neq i}\lambda_{j}\left({\bf C}_{t}\left(y\right)\right)\nonumber \\
 & \le\sqrt{N}\det\left({\bf C}_{t}\left(y\right)\right),
\end{flalign}
and then (\ref{eq:Elem_Lemma3_1}) becomes
\begin{align}
\left\Vert {\bf C}_{t}^{-1}\left(x\right)-{\bf C}_{t}^{-1}\left(y\right)\right\Vert _{F} & \le\dfrac{\left\Vert \mathrm{adj}\left({\bf C}_{t}\left(x\right)\right)-\mathrm{adj}\left({\bf C}_{t}\left(y\right)\right)\right\Vert _{F}}{\det\left({\bf C}_{t}\left(x\right)\right)}+\sqrt{N}\dfrac{\left|\det\left({\bf C}_{t}\left(x\right)\right)-\det\left({\bf C}_{t}\left(y\right)\right)\right|}{\det\left({\bf C}_{t}\left(x\right)\right)}\nonumber \\
 & \le\dfrac{\left\Vert \mathrm{adj}\left({\bf C}_{t}\left(x\right)\right)-\mathrm{adj}\left({\bf C}_{t}\left(y\right)\right)\right\Vert _{F}}{\lambda_{inf}^{N}}+\dfrac{N^{3}K_{DET}K_{\boldsymbol{\Sigma}}}{\lambda_{inf}^{N}}\left|x-y\right|.\label{eq:Elem_Lemma3_2}
\end{align}
Next, the numerator of the first fraction from the left may be expressed
as
\begin{flalign}
\left\Vert \mathrm{adj}\left({\bf C}_{t}\left(x\right)\right)-\mathrm{adj}\left({\bf C}_{t}\left(y\right)\right)\right\Vert _{F} & \equiv\sqrt{\sum_{\left(i,j\right)\in\mathbb{N}_{N}^{+}\times\mathbb{N}_{N}^{+}}\left(\mathrm{adj}\left({\bf C}_{t}\left(x\right)\right)^{ij}-\mathrm{adj}\left({\bf C}_{t}\left(y\right)\right)^{ij}\right)^{2}}\nonumber \\
 & \equiv\sqrt{\sum_{\left(i,j\right)\in\mathbb{N}_{N}^{+}\times\mathbb{N}_{N}^{+}}\left(\left(-1\right)^{i+j}\left[{\cal M}_{ij}\left({\bf C}_{t}\left(x\right)\right)-{\cal M}_{ij}\left({\bf C}_{t}\left(y\right)\right)\right]\right)^{2}}\nonumber \\
 & \equiv\sqrt{\sum_{\left(i,j\right)\in\mathbb{N}_{N}^{+}\times\mathbb{N}_{N}^{+}}\left({\cal M}_{ij}\left({\bf C}_{t}\left(x\right)\right)-{\cal M}_{ij}\left({\bf C}_{t}\left(y\right)\right)\right)^{2}},
\end{flalign}
where ${\cal M}_{ij}\left({\bf C}_{t}\left(x\right)\right)$ denotes
the $\left(i,j\right)$-th minor of ${\bf C}_{t}\left(x\right)$,
which constitutes the \textit{determinant} of the $\left(N-1\right)\times\left(N-1\right)$
matrix formulated by removing the $i$-th row and the $j$-th column
of ${\bf C}_{t}\left(x\right)$. Consequently, from Lemma \ref{Elementary_Corollary_2},
there exists a constant $K_{det}$, possibly dependent on $N$, such
that, $\forall t\in\mathbb{N}$,
\begin{align}
\left\Vert \mathrm{adj}\left({\bf C}_{t}\left(x\right)\right)-\mathrm{adj}\left({\bf C}_{t}\left(y\right)\right)\right\Vert _{F} & \le\sqrt{\sum_{\left(i,j\right)\in\mathbb{N}_{N}^{+}\times\mathbb{N}_{N}^{+}}N^{4}K_{det}^{2}K_{\boldsymbol{\Sigma}}^{2}\left|x-y\right|^{2}},
\end{align}
or, equivalently,
\begin{align}
\left\Vert \mathrm{adj}\left({\bf C}_{t}\left(x\right)\right)-\mathrm{adj}\left({\bf C}_{t}\left(y\right)\right)\right\Vert _{F} & \le N^{3}K_{det}K_{\boldsymbol{\Sigma}}\left|x-y\right|,
\end{align}
$\forall\left(x,y\right)\in{\cal Z}\times{\cal Z}.$ Therefore, combining
with (\ref{eq:Elem_Lemma3_2}), we get
\begin{flalign}
\left\Vert {\bf C}_{t}^{-1}\left(x\right)-{\bf C}_{t}^{-1}\left(y\right)\right\Vert _{F} & \le\dfrac{N^{3}}{\lambda_{inf}^{N}}\left(K_{DET}+K_{det}\right)K_{\boldsymbol{\Sigma}}\left|x-y\right|\nonumber \\
 & \le\dfrac{27\lambda_{inf}^{-3/\log\left(\lambda_{inf}\right)}}{\left(\log\left(\lambda_{inf}\right)\right)^{3}}\left(K_{DET}+K_{det}\right)K_{\boldsymbol{\Sigma}}\left|x-y\right|\nonumber \\
 & \triangleq K_{INV}\left|x-y\right|,
\end{flalign}
and the proof is complete.
\end{proof}
Next, we state the following simple probabilistic result, related
to the expansiveness of the norm of the observation vector in a stochastic
sense, under both base measures ${\cal P}$ and $\widetilde{{\cal P}}$
considered throughout the paper (see Section II.B).
\begin{lem}
\label{Lemma_WHP}Consider the random quadratic form
\begin{equation}
Q_{t}\left(\omega\right)\triangleq\left\Vert {\bf y}_{t}\left(\omega\right)\right\Vert _{2}^{2}\equiv\left\Vert \overline{{\bf y}}_{t}\left(X_{t}\left(\omega\right)\right)+\boldsymbol{\mu}_{t}\left(X_{t}\left(\omega\right)\right)\right\Vert _{2}^{2},\quad t\in\mathbb{N}.
\end{equation}
Then, for any fixed $t\in\mathbb{N}$ and any freely chosen $C\ge1$,
there exists a bounded constant $\gamma>1$, such that the measurable
set
\begin{equation}
{\cal T}_{t}\triangleq\left\{ \omega\in\Omega\left|\sup_{i\in\mathbb{N}_{t}}Q_{i}\left(\omega\right)<\gamma CN\left(1+\log\left(t+1\right)\right)\right.\right\} 
\end{equation}
satisfies
\begin{equation}
\min\left\{ {\cal P}\left({\cal T}_{t}\right),\widetilde{{\cal P}}\left({\cal T}_{t}\right)\right\} \ge1-\dfrac{\exp\left(-CN\right)}{\left(t+1\right)^{CN-1}},
\end{equation}
that is, the sequence of quadratic forms $\left\{ Q_{i}\left(\omega\right)\right\} _{i\in\mathbb{N}_{t}}$
is uniformly bounded with very high probability under both base measures
${\cal P}$ and $\widetilde{{\cal P}}$.\end{lem}
\begin{proof}[Proof of Lemma \ref{Lemma_WHP}]
First, it is true that
\begin{flalign}
\left\Vert {\bf y}_{t}\left(\omega\right)\right\Vert _{2}^{2} & \equiv\left\Vert \overline{{\bf y}}_{t}\left(X_{t}\left(\omega\right)\right)+\boldsymbol{\mu}_{t}\left(X_{t}\left(\omega\right)\right)\right\Vert _{2}^{2}\nonumber \\
 & \equiv\left\Vert \overline{{\bf y}}_{t}\left(X_{t}\left(\omega\right)\right)\right\Vert _{2}^{2}+2{\bf y}_{t}^{\boldsymbol{T}}\left(X_{t}\left(\omega\right)\right)\boldsymbol{\mu}_{t}\left(X_{t}\left(\omega\right)\right)+\left\Vert \boldsymbol{\mu}_{t}\left(X_{t}\left(\omega\right)\right)\right\Vert _{2}^{2}\nonumber \\
 & \le\left\Vert \overline{{\bf y}}_{t}\left(X_{t}\left(\omega\right)\right)\right\Vert _{2}^{2}+2\left\Vert \overline{{\bf y}}_{t}\left(X_{t}\left(\omega\right)\right)\right\Vert _{2}\mu_{sup}+\mu_{sup}^{2}.
\end{flalign}
Also, under ${\cal P}$, for each $t\in\mathbb{N}$, the random variable
$\overline{{\bf y}}_{t}\left(X_{t}\right)$ constitutes an $N$-dimensional,
conditionally (on $X_{t}$) Gaussian random variable with zero mean
and covariance matrix ${\bf C}_{t}\left(X_{t}\right)$, that is
\begin{equation}
\overline{{\bf y}}_{t}\left|X_{t}\right.\sim{\cal N}\left(0,{\bf C}_{t}\left(X_{t}\right)\equiv{\bf C}_{\overline{{\bf y}}_{t}\left|X_{t}\right.}\right).
\end{equation}
Then, \textit{if $X_{t}$ is given}, 
\begin{equation}
\overline{Q}_{t}\left(\omega\right)\triangleq\left\Vert \overline{{\bf y}}_{t}\left(X_{t}\left(\omega\right)\right)\right\Vert _{2}^{2}
\end{equation}
can be shown to admit the very useful alternative representation (for
instance, see \cite{Mathai1992quadratic}, pp. 89 - 90)
\begin{gather}
\overline{Q}_{t}\equiv\sum_{j\in\mathbb{N}_{N}^{+}}\lambda_{j}\left({\bf C}_{t}\left(X_{t}\right)\right)U_{j}^{2},\quad\forall t\in\mathbb{N},\quad\text{with}\label{eq:Q_Representation-1}\\
\left\{ U_{j}\right\} _{j\in\mathbb{N}_{N}^{+}}\overset{i.i.d.}{\sim}{\cal N}\left(0,1\right).
\end{gather}
From (\ref{eq:Q_Representation-1}), one can readily observe that
the statistical dependence of $\overline{Q}_{t}$ on $X_{t}$ concentrates
only on the eigenvalues of the covariance matrix ${\bf C}_{t}\left(X_{t}\right)$,
for which we have already assumed the existence of a finite supremum
explicitly (see Assumption 1). Consequently, \textit{conditioning
on the process $X_{t}$}, we can bound (\ref{eq:Q_Representation-1})
as
\begin{equation}
\overline{Q}_{t}\le\lambda_{sup}\sum_{j\in\mathbb{N}_{N}^{+}}U_{j}^{2}\triangleq\lambda_{sup}U,\quad\text{with}\quad U\sim\chi^{2}\left(N\right),
\end{equation}
almost everywhere and everywhere in time, where the RHS is independent
of $X_{t}$. Next, from (\cite{LaurentMassart2000}, p. 1325), we
know that for any chi squared random variable $U$ with $N$ degrees
of freedom,
\begin{equation}
{\cal P}\left(U\ge N+2\sqrt{Nu}+2u\right)\le\exp\left(-u\right),\quad\forall u>0.
\end{equation}
Setting $u\equiv CN\left(1+\log\left(t+1\right)\right)$ for any $C\ge1$
and any $t\in\mathbb{N}$,
\begin{align}
{\cal P}\left(\vphantom{\sqrt{C\left(1+\log\left(t+1\right)\right)}}U\ge N+2N\sqrt{C\left(1+\log\left(t+1\right)\right)}\right.+\left.2CN\left(1+\log\left(t+1\right)\right)\vphantom{\sqrt{C\left(1+\log\left(t+1\right)\right)}}\right) & \le\dfrac{\exp\left(-CN\right)}{\left(t+1\right)^{CN}}.
\end{align}
a statement which equivalently means that, with probability at least
$1-\left(t+1\right)^{-CN}\exp\left(-CN\right)$,
\begin{align}
U & <N+2N\sqrt{C\left(1+\log\left(t+1\right)\right)}+2CN\left(1+\log\left(t+1\right)\right).
\end{align}
However, because the RHS of the above inequality is upper bounded
by $5CN\left(1+\log\left(t+1\right)\right)$, 
\begin{multline}
{\cal P}\left(U<5CN\left(1+\log\left(t+1\right)\right)\right)\\
\ge{\cal P}\left(\vphantom{\sqrt{C\left(1+\log\left(t+1\right)\right)}}U<N+2N\sqrt{C\left(1+\log\left(t+1\right)\right)}\right.+\left.2CN\left(1+\log\left(t+1\right)\right)\vphantom{\sqrt{C\left(1+\log\left(t+1\right)\right)}}\right)\ge1-\dfrac{\exp\left(-CN\right)}{\left(t+1\right)^{CN}}.
\end{multline}
Hence, $\forall i\in\mathbb{N}_{t}$,
\begin{flalign}
{\cal P}\left(\left.\overline{Q}_{i}\ge5\lambda_{sup}CN\left(1+\log\left(t+1\right)\right)\right|X_{i}\right) & \le{\cal P}\left(U\ge5CN\left(1+\log\left(t+1\right)\right)\right)\nonumber \\
 & \le\dfrac{\exp\left(-CN\right)}{\left(t+1\right)^{CN}},
\end{flalign}
and, thus,
\begin{flalign}
{\cal P}\left(\overline{Q}_{i}\ge5\lambda_{sup}CN\left(1+\log\left(t+1\right)\right)\right) & =\int{\cal P}\left(\left.\overline{Q}_{i}\ge5\lambda_{sup}CN\left(1+\log\left(t+1\right)\right)\right|X_{i}\right)\mathrm{d}{\cal P}_{X_{i}}\nonumber \\
 & \le\dfrac{\exp\left(-CN\right)}{\left(t+1\right)^{CN}}\int\mathrm{d}{\cal P}_{X_{i}}\equiv\dfrac{\exp\left(-CN\right)}{\left(t+1\right)^{CN}}.
\end{flalign}
However, we would like to produce a bound on the supremum of all the
$Q_{i},i\in\mathbb{N}_{t}$. Indeed, using the naive union bound,
\begin{flalign*}
{\cal P}\left(\bigcup_{i\in\mathbb{N}_{t}}\left\{ \overline{Q}_{i}\ge5\lambda_{sup}CN\left(1+\log\left(t+1\right)\right)\right\} \right) & \le\sum_{i\in\mathbb{N}_{t}}{\cal P}\left(\overline{Q}_{i}\ge\lambda_{sup}5CN\left(1+\log\left(t+1\right)\right)\right)\\
 & \le\dfrac{\left(t+1\right)\exp\left(-CN\right)}{\left(t+1\right)^{CN}}\equiv\dfrac{\exp\left(-CN\right)}{\left(t+1\right)^{CN-1}}
\end{flalign*}
or, equivalently, 
\begin{flalign}
{\cal P}\left(\sup_{i\in\mathbb{N}_{t}}\overline{Q}_{i}<5\lambda_{sup}CN\left(1+\log\left(t+1\right)\right)\right) & \equiv{\cal P}\left(\left\{ \overline{Q}_{i}<5\lambda_{sup}CN\left(1+\log\left(t+1\right)\right),\forall i\in\mathbb{N}_{t}\right\} \right)\nonumber \\
 & \equiv{\cal P}\left(\bigcap_{i\in\mathbb{N}_{t}}\left\{ \overline{Q}_{i}<5\lambda_{sup}CN\left(1+\log\left(t+1\right)\right)\right\} \right)\nonumber \\
 & \ge1-\dfrac{\exp\left(-CN\right)}{\left(t+1\right)^{CN-1}},
\end{flalign}
holding true $\forall t\in\mathbb{N}$. Consequently, working in the
same fashion as above, it is true that, with at least the same probability
of success,
\begin{flalign}
\sup_{i\in\mathbb{N}_{t}}Q_{i}\left(\omega\right) & <5\lambda_{sup}CN\left(1+\log\left(t+1\right)\right)+2\sqrt{5\lambda_{sup}CN\left(1+\log\left(t+1\right)\right)}\mu_{sup}+\mu_{sup}^{2}\nonumber \\
 & <5\lambda_{sup}\left(1+2\mu_{sup}+\mu_{sup}^{2}\right)CN\left(1+\log\left(t+1\right)\right)
\end{flalign}
or, setting $\gamma_{1}\triangleq5\lambda_{sup}\left(1+\mu_{sup}\right)^{2}>1$,
\begin{equation}
\sup_{i\in\mathbb{N}_{t}}Q_{i}\left(\omega\right)<\gamma_{1}CN\left(1+\log\left(t+1\right)\right).
\end{equation}

Now, under the alternative base measure $\widetilde{{\cal P}}$, ${\bf y}_{t}$
constitutes a Gaussian vector white noise process with zero mean and
covariance matrix the identity, statistically independent of the process
$X_{t}$ (see Theorem 2). That is, for each $t$, the elements of
${\bf y}_{t}$ are themselves independent to each other. Thus, for
all $t\in\mathbb{N}$ and for all $i\in\mathbb{N}_{t}$ and using
similar arguments as the ones made above, it should be true that 
\begin{equation}
\widetilde{{\cal P}}\left(Q_{i}<5CN\left(1+\log\left(t+1\right)\right)\right)\ge1-\dfrac{\exp\left(-CN\right)}{\left(t+1\right)^{CN}}
\end{equation}
and taking the union bound, we end up with the inequality
\begin{align}
\widetilde{{\cal P}}\left(\sup_{i\in\mathbb{N}_{t}}Q_{i}<5CN\left(1+\log\left(t+1\right)\right)\right) & \ge1-\dfrac{\exp\left(-CN\right)}{\left(t+1\right)^{CN-1}}.
\end{align}
Defining $\gamma\triangleq\max\left\{ \gamma_{1},5\right\} \equiv\gamma_{1}$,
it must be true that, for all $t\in\mathbb{N}$, 
\begin{multline}
\min\left\{ {\cal P}\left(\sup_{i\in\mathbb{N}_{t}}Q_{i}<\gamma CN\left(1+\log\left(t+1\right)\right)\right),\widetilde{{\cal P}}\left(\sup_{i\in\mathbb{N}_{t}}Q_{i}<\gamma CN\left(1+\log\left(t+1\right)\right)\right)\right\} \\
\ge1-\dfrac{\exp\left(-CN\right)}{\left(t+1\right)^{CN-1}},
\end{multline}
therefore completing the proof of the lemma.
\end{proof}
Continuing our presentation of preliminary results towards the proof
of Theorem (\ref{CONVERGENCE_THEOREM}) and leveraging the power of
${\cal C}$-weak convergence and Lemma \ref{Portmanteau-1}, let us
present the following lemma, connecting ${\cal C}$-weak convergence
of random variables with convergence in the ${\cal L}_{1}$ sense.
\begin{lem}
\noindent \label{L1WEAK_Lemma}\textbf{\textup{(From ${\cal C}$-Weak
Convergence to Convergence in ${\cal L}_{1}$)}} Consider the sequence
of discrete time stochastic processes $\left\{ X_{t}^{\mathtt{A}}\right\} _{\mathtt{A}\in\mathbb{N}}$,
as well as a ``limit'' process $X_{t}$,$t\in\mathbb{N}$, all being
$\left(\mathbb{R},\mathscr{S}\triangleq\mathscr{B}\left(\mathbb{R}\right)\right)$-valued
and all defined on a common base space $\left(\Omega,\mathscr{F},{\cal P}\right)$.
Further, suppose that all members of the collection $\left\{ \left\{ X_{t}^{\mathtt{A}}\right\} _{\mathtt{A}\in\mathbb{N}},X_{t}\right\} _{t\in\mathbb{N}}$
are almost surely bounded in ${\cal Z}\equiv\left[a,b\right]$ (with
$-\infty<a<b<\infty$) and that
\begin{equation}
{\cal P}_{\left.X_{t}^{\mathtt{A}}\right|X_{t}}^{\mathtt{A}}\left(\left.\cdot\right|X_{t}\right)\stackrel[\mathtt{A}\rightarrow\infty]{{\cal W}}{\longrightarrow}\delta_{X_{t}}\left(\cdot\right)\equiv\mathds{1}_{\left(\cdot\right)}\left(X_{t}\right),\quad\forall t\in\mathbb{N},\label{eq:C_weak_COND}
\end{equation}
that is, the sequence $\left\{ X_{t}^{\mathtt{A}}\right\} _{\mathtt{A}\in\mathbb{N}}$
is marginally ${\cal C}$-weakly convergent to $X_{t}$, given $X_{t}$,
for all $t$. Then, it is true that
\begin{equation}
\mathbb{E}\left\{ \left|X_{t}-X_{t}^{\mathtt{A}}\right|\right\} \underset{\mathtt{A}\rightarrow\infty}{\longrightarrow}0,\quad\forall t\in\mathbb{N},
\end{equation}
or, equivalently, $X_{t}^{\mathtt{A}}\stackrel[\mathtt{A}\rightarrow\infty]{{\cal L}_{1}}{\longrightarrow}X_{t}$,
for all $t$.\end{lem}
\begin{proof}[Proof of Lemma \ref{L1WEAK_Lemma}]
Let all the hypotheses of Lemma \ref{L1WEAK_Lemma} hold true. Then,
we know that, $\forall t\in\mathbb{N}$,
\begin{equation}
{\displaystyle \lim_{n\rightarrow\infty}}{\cal P}_{\left.X_{t}^{\mathtt{A}}\right|X_{t}}^{\mathtt{A}}\left(\left.{\cal A}\right|X_{t}\left(\omega\right)\right)=\delta_{X_{t}\left(\omega\right)}\left({\cal A}\right),\quad{\cal P}-a.e.,
\end{equation}
for all continuity Borel sets ${\cal A}\in\mathscr{S}$. Using the
tower property, it is also true that
\begin{equation}
\mathbb{E}\left\{ \left|X_{t}-X_{t}^{\mathtt{A}}\right|\right\} \equiv\mathbb{E}\left\{ \mathbb{E}\left\{ \left.\left|X_{t}-X_{t}^{\mathtt{A}}\right|\right|\sigma\left\{ X_{t}\right\} \right\} \right\} .
\end{equation}
Therefore, in order to show that $X_{t}^{\mathtt{A}}\stackrel[\mathtt{A}\rightarrow\infty]{{\cal L}_{1}}{\longrightarrow}X_{t}$
for each $t\in\mathbb{N}$, it suffices to show that
\begin{equation}
\mathbb{E}\left\{ \left.\left|X_{t}-X_{t}^{\mathtt{A}}\right|\right|\sigma\left\{ X_{t}\right\} \right\} \left(\omega\right)\stackrel[\mathtt{A}\rightarrow\infty]{a.s.}{\longrightarrow}0,\quad\forall t\in\mathbb{N}.
\end{equation}
Then, the Dominated Convergence Theorem would produce the desired
result.

Of course, because all members of the collection $\left\{ \left\{ X_{t}^{\mathtt{A}}\right\} _{\mathtt{A}\in\mathbb{N}},X_{t}\right\} _{t\in\mathbb{N}}$
are almost surely bounded in ${\cal Z}$, all members of the collection
$\left\{ \left\{ \left|X_{t}-X_{t}^{\mathtt{A}}\right|\right\} _{\mathtt{A}\in\mathbb{N}}\right\} _{t\in\mathbb{N}}$
must be bounded almost surely in the compact set ${\cal \widehat{Z}}\triangleq\left[0,2\delta\right]\subset\mathbb{R}$,
where $\delta\triangleq\max\left\{ \left|a\right|,\left|b\right|\right\} $.

Let us define the continuous and bounded function
\begin{equation}
f\left(x\right)\triangleq\begin{cases}
x, & \text{if }x\in{\cal \widehat{Z}}\\
2\delta, & \text{if }x>2\delta\\
0, & \text{if }x<0
\end{cases}.
\end{equation}
Then, from Lemma \ref{Portmanteau-1} and using conditional probability
measures it must be true that for each $t\in\mathbb{N}$, a version
of the conditional expectation of interest is explicitly given by
\begin{flalign}
\mathbb{E}\left\{ \left.f\left(\left|X_{t}-X_{t}^{\mathtt{A}}\right|\right)\right|\sigma\left\{ X_{t}\right\} \right\} \left(\omega\right) & \equiv\int f\left(\left|X_{t}\left(\omega\right)-x\right|\right){\cal P}_{\left.X_{t}^{\mathtt{A}}\right|X_{t}}^{\mathtt{A}}\left(\mathrm{d}x\left|X_{t}\left(\omega\right)\right.\right)\underset{\mathtt{A}\rightarrow\infty}{\longrightarrow}\nonumber \\
 & \underset{\mathtt{A}\rightarrow\infty}{\longrightarrow}\int f\left(\left|X_{t}\left(\omega\right)-x\right|\right)\delta_{X_{t}\left(\omega\right)}\left(\mathrm{d}x\right)\equiv0,\quad{\cal P}-a.e.,
\end{flalign}
since, for each $\omega\in\Omega$, $X_{t}\left(\omega\right)$ is
constant. Further, by definition of $f$, 
\begin{flalign}
\mathbb{E}\left\{ \left.f\left(\left|X_{t}-X_{t}^{\mathtt{A}}\right|\right)\right|\sigma\left\{ X_{t}\right\} \right\} \left(\omega\right) & \equiv\mathbb{E}\left\{ \left.\left|X_{t}-X_{t}^{\mathtt{A}}\right|\mathds{1}_{\left(\left|X_{t}-X_{t}^{\mathtt{A}}\right|\right)}\left({\cal \widehat{Z}}\right)\right|\sigma\left\{ X_{t}\right\} \right\} \left(\omega\right)\nonumber \\
 & \equiv\mathbb{E}\left\{ \left.\left|X_{t}-X_{t}^{\mathtt{A}}\right|\right|\sigma\left\{ X_{t}\right\} \right\} \left(\omega\right),\quad{\cal P}-a.e.,
\end{flalign}
and for all $t\in\mathbb{N}$, which means that
\begin{equation}
\mathbb{E}\left\{ \left.\left|X_{t}-X_{t}^{\mathtt{A}}\right|\right|\sigma\left\{ X_{t}\right\} \right\} \left(\omega\right)\stackrel[\mathtt{A}\rightarrow\infty]{a.s.}{\longrightarrow}0,\quad\forall t\in\mathbb{N}.
\end{equation}
Calling dominated convergence proves the result.
\end{proof}
Additionally, the following useful (to us) result is also true. The
proof, being elementary, is omitted.
\begin{lem}
\label{L1WEAK_LEMMA}\textbf{\textup{(Convergence of the Supremum)}}
Pick any natural $T<\infty$. If, under any circumstances,
\begin{equation}
\mathbb{E}\left\{ \left|X_{t}-X_{t}^{\mathtt{A}}\right|\right\} \underset{\mathtt{A}\rightarrow\infty}{\longrightarrow}0,\quad\forall t\in\mathbb{N}_{T},
\end{equation}
 then
\begin{equation}
\sup_{t\in\mathbb{N}_{T}}\mathbb{E}\left\{ \left|X_{t}-X_{t}^{\mathtt{A}}\right|\right\} \underset{\mathtt{A}\rightarrow\infty}{\longrightarrow}0.
\end{equation}

\end{lem}

\subsection{The Key Lemma}

We are now ready to present our key lemma, which will play an important
role in establishing our main result (Theorem \ref{CONVERGENCE_THEOREM})
later on. For proving this result, we make use of all the intermediate
ones presented in the previous subsections.
\begin{lem}
\noindent \label{RD_Derivatives_Lemma}\textbf{\textup{(Convergence
of the Likelihoods)}} Consider the stochastic process
\begin{equation}
\widehat{\Lambda}_{t}\triangleq\dfrac{\exp\left(-\dfrac{1}{2}{\displaystyle \sum_{i\in\mathbb{N}_{t}}}\overline{{\bf y}}_{i}^{\boldsymbol{T}}\left(X_{i}\right){\bf C}_{i}^{-1}\left(X_{i}\right)\overline{{\bf y}}_{i}\left(X_{i}\right)\right)}{{\displaystyle \prod_{i\in\mathbb{N}_{t}}}\sqrt{\det\left({\bf C}_{i}\left(X_{i}\right)\right)}}\triangleq\dfrac{\mathfrak{N}_{t}}{\mathfrak{D}_{t}},\quad t\in\mathbb{N}.
\end{equation}
Consider also the process $\widehat{\Lambda}_{t}^{\mathtt{A}}\triangleq\mathfrak{N}_{t}^{\mathtt{A}}/\mathfrak{D}_{t}^{\mathtt{A}}$,
defined exactly the same way as $\widehat{\Lambda}_{t}$, but replacing
$X_{i}$ with the approximation $X_{i}^{\mathtt{A}}$, $\forall i\in\mathbb{N}_{t}$.
Further, pick any natural $T<\infty$ and suppose either of the following:
\begin{itemize}
\item \noindent For all $t\in\mathbb{N}_{T}$, the sequence $\left\{ X_{t}^{\mathtt{A}}\right\} _{\mathtt{A}\in\mathbb{N}}$
is marginally ${\cal C}$-weakly convergent to $X_{t}$, given $X_{t}$,
that is,
\begin{equation}
{\cal P}_{\left.X_{t}^{\mathtt{A}}\right|X_{t}}^{\mathtt{A}}\left(\left.\cdot\right|X_{t}\right)\stackrel[\mathtt{A}\rightarrow\infty]{{\cal W}}{\longrightarrow}\delta_{X_{t}}\left(\cdot\right),\quad\forall t\in\mathbb{N}_{T}.
\end{equation}

\item \noindent For all $t\in\mathbb{N}_{T}$, the sequence $\left\{ X_{t}^{\mathtt{A}}\right\} _{\mathtt{A}\in\mathbb{N}}$
is marginally convergent to $X_{t}$ in probability, that is,
\begin{equation}
X_{t}^{\mathtt{A}}\stackrel[\mathtt{A}\rightarrow\infty]{{\cal P}}{\longrightarrow}X_{t},\quad\forall t\in\mathbb{N}_{T}.
\end{equation}

\end{itemize}
\noindent Then, there exists a measurable subset $\widehat{\Omega}_{T}\subseteq\Omega$,
such that
\begin{equation}
\lim_{\mathtt{A}\rightarrow\infty}\sup_{t\in\mathbb{N}_{T}}\sup_{\omega\in\widehat{\Omega}_{T}}\mathbb{E}_{\widetilde{{\cal P}}}\left\{ \left.\left|\widehat{\Lambda}_{t}-\widehat{\Lambda}_{t}^{\mathtt{A}}\right|\right|\mathscr{Y}_{t}\right\} \left(\omega\right)\equiv0,
\end{equation}
where the ${\cal P}$,$\widetilde{{\cal P}}$-measures of $\widehat{\Omega}_{T}$
satisfy 
\begin{equation}
\min\left\{ {\cal P}\left(\widehat{\Omega}_{T}\right),\widetilde{{\cal P}}\left(\widehat{\Omega}_{T}\right)\right\} \ge1-\dfrac{\exp\left(-CN\right)}{\left(T+1\right)^{CN-1}},
\end{equation}
for any free but finite constant $C\ge1$.
\end{lem}
\medskip{}

\begin{proof}[Proof of Lemma \ref{RD_Derivatives_Lemma}]
From Lemma \ref{Elementary_Lemma_1}, it is true that
\begin{flalign}
\left|\widehat{\Lambda}_{t}-\widehat{\Lambda}_{t}^{\mathtt{A}}\right|\equiv\left|\dfrac{\mathfrak{N}_{t}}{\mathfrak{D}_{t}}-\dfrac{\mathfrak{N}_{t}^{\mathtt{A}}}{\mathfrak{D}_{t}^{\mathtt{A}}}\right| & \le\dfrac{\left|\mathfrak{N}_{t}-\mathfrak{N}_{t}^{\mathtt{A}}\right|}{\left|\mathfrak{D}_{t}^{\mathtt{A}}\right|}+\left|\mathfrak{N}_{t}\right|\left|\dfrac{1}{\mathfrak{D}_{t}}-\dfrac{1}{\mathfrak{D}_{t}^{\mathtt{A}}}\right|\nonumber \\
 & \le\dfrac{\left|\mathfrak{N}_{t}-\mathfrak{N}_{t}^{\mathtt{A}}\right|}{\left|\mathfrak{D}_{t}^{\mathtt{A}}\right|}+\left|\dfrac{1}{\mathfrak{D}_{t}}-\dfrac{1}{\mathfrak{D}_{t}^{\mathtt{A}}}\right|.\label{eq:CORE_Inequality}
\end{flalign}
We first concentrate on the determinant part (second term) of the
RHS of (\ref{eq:CORE_Inequality}). Directly invoking Lemma \ref{Elementary_Lemma_2},
it will be true that
\begin{flalign}
 & \hspace{-2pt}\hspace{-2pt}\left|\dfrac{1}{\mathfrak{D}_{t}}-\dfrac{1}{\mathfrak{D}_{t}^{\mathtt{A}}}\right|\nonumber \\
 & \hspace{-2pt}\equiv\hspace{-2pt}\left|\prod_{i=0}^{t}\dfrac{1}{\sqrt{\det\left({\bf C}_{i}\left(X_{i}\right)\right)}}-\prod_{i=0}^{t}\dfrac{1}{\sqrt{\det\left(\hspace{-2pt}{\bf C}_{i}\left(X_{i}^{\mathtt{A}}\right)\hspace{-2pt}\right)}}\right|\nonumber \\
 & \hspace{-2pt}\le\hspace{-2pt}\sum_{i=0}^{t}\hspace{-2pt}\left(\prod_{j=0}^{i-1}\dfrac{1}{\sqrt{\det\left({\bf C}_{j}\left(X_{j}\right)\right)}}\right)\hspace{-2pt}\hspace{-2pt}\left(\prod_{j=i+1}^{t}\dfrac{1}{\sqrt{\det\left(\hspace{-2pt}{\bf C}_{j}\left(X_{j}^{\mathtt{A}}\right)\hspace{-2pt}\right)}}\right)\hspace{-2pt}\dfrac{\left|\sqrt{\det\left({\bf C}_{i}\left(X_{i}\right)\right)}\hspace{-2pt}-\hspace{-2pt}\sqrt{\det\left(\hspace{-2pt}{\bf C}_{i}\left(X_{i}^{\mathtt{A}}\right)\hspace{-2pt}\right)}\right|}{\sqrt{\det\left({\bf C}_{i}\left(X_{i}\right)\right)\det\left(\hspace{-2pt}{\bf C}_{i}\left(X_{i}^{\mathtt{A}}\right)\hspace{-2pt}\right)}}\nonumber \\
 & \hspace{-2pt}=\hspace{-2pt}\sum_{i=0}^{t}\hspace{-2pt}\left(\prod_{j=0}^{i-1}\dfrac{1}{\sqrt{{\displaystyle \prod_{n=1}^{N}}\lambda_{n}\left({\bf C}_{j}\left(X_{j}\right)\right)}}\right)\hspace{-2pt}\hspace{-2pt}\left(\prod_{j=i+1}^{t}\dfrac{1}{\sqrt{{\displaystyle \prod_{n=1}^{N}}\lambda_{n}\left(\hspace{-2pt}{\bf C}_{j}\left(X_{j}^{\mathtt{A}}\right)\hspace{-2pt}\right)}}\right)\hspace{-2pt}\dfrac{\left|\sqrt{\det\left({\bf C}_{i}\left(X_{i}\right)\right)}\hspace{-2pt}-\hspace{-2pt}\sqrt{\det\left(\hspace{-2pt}{\bf C}_{i}\left(X_{i}^{\mathtt{A}}\right)\hspace{-2pt}\right)}\right|}{\sqrt{{\displaystyle \prod_{n=1}^{N}}\lambda_{n}\left({\bf C}_{j}\left(X_{j}\right)\right)\lambda_{n}\left(\hspace{-2pt}{\bf C}_{j}\left(X_{j}^{\mathtt{A}}\right)\hspace{-2pt}\right)}}\nonumber \\
 & \hspace{-2pt}\le\hspace{-2pt}\sum_{i=0}^{t}\dfrac{1}{2\lambda_{inf}^{Ni/2}\lambda_{inf}^{N\left(t-i\right)/2}}\dfrac{\left|\det\left({\bf C}_{i}\left(X_{i}\right)\right)-\det\left(\hspace{-2pt}{\bf C}_{i}\left(X_{i}^{\mathtt{A}}\right)\hspace{-2pt}\right)\right|}{\lambda_{inf}^{N}}\nonumber \\
 & \hspace{-2pt}\equiv\hspace{-2pt}\dfrac{1}{2\sqrt{\lambda_{inf}^{N\left(t+2\right)}}}\sum_{i=0}^{t}\left|\det\left({\bf C}_{i}\left(X_{i}\right)\right)-\det\left(\hspace{-2pt}{\bf C}_{i}\left(X_{i}^{\mathtt{A}}\right)\hspace{-2pt}\right)\right|.
\end{flalign}
From Lemma \ref{Elementary_Corollary_2}, we can bound the RHS of
the above expression as
\begin{flalign}
\left|\dfrac{1}{\mathfrak{D}_{t}}-\dfrac{1}{\mathfrak{D}_{t}^{\mathtt{A}}}\right| & \le\dfrac{NK_{DET}}{2\sqrt{\lambda_{inf}^{N\left(t+2\right)}}}\sum_{i=0}^{t}\left\Vert {\bf C}_{i}\left(X_{i}\right)-{\bf C}_{i}\left(X_{i}^{\mathtt{A}}\right)\right\Vert _{F}\nonumber \\
 & \equiv\dfrac{NK_{DET}}{2\sqrt{\lambda_{inf}^{N\left(t+2\right)}}}\sum_{i=0}^{t}\left\Vert \boldsymbol{\Sigma}_{i}\left(X_{i}\right)-\boldsymbol{\Sigma}_{i}\left(X_{i}^{\mathtt{A}}\right)\right\Vert _{F}.\label{eq:DET_LIP}
\end{flalign}
And from Lemma \ref{Elementary_Corollary_1}, (\ref{eq:DET_LIP})
becomes
\begin{equation}
\left|\dfrac{1}{\mathfrak{D}_{t}}-\dfrac{1}{\mathfrak{D}_{t}^{\mathtt{A}}}\right|\le\dfrac{N^{2}K_{DET}K_{\boldsymbol{\Sigma}}}{2\sqrt{\lambda_{inf}^{N\left(t+2\right)}}}\sum_{i=0}^{t}\left|X_{i}-X_{i}^{\mathtt{A}}\right|.\label{eq:DET_star}
\end{equation}

We now turn our attention to the ``difference of exponentials''
part (first term) of the RHS of (\ref{eq:CORE_Inequality}). First,
we know that
\begin{flalign}
{\displaystyle \prod_{i=0}^{t}}\det\left({\bf C}_{i}\left(X_{i}^{\mathtt{A}}\right)\right) & \ge{\displaystyle \prod_{i=0}^{t}}{\displaystyle \prod_{j=1}^{N}}\lambda_{inf}\equiv\lambda_{inf}^{N\left(t+1\right)},
\end{flalign}
yielding the inequality
\begin{equation}
\dfrac{\left|\mathfrak{N}_{t}-\mathfrak{N}_{t}^{\mathtt{A}}\right|}{\left|\mathfrak{D}_{t}^{\mathtt{A}}\right|}\le\dfrac{\left|\mathfrak{N}_{t}-\mathfrak{N}_{t}^{\mathtt{A}}\right|}{\sqrt{\lambda_{inf}^{N\left(t+1\right)}}},\label{eq:Diff_of EXPS}
\end{equation}
where $\lambda_{inf}>1$ (see Assumption 1). Next, making use of the
inequality \cite{Kushner2008}
\begin{equation}
\left|\exp\left(\alpha\right)-\exp\left(\beta\right)\right|\le\left|\alpha-\beta\right|\left(\exp\left(\alpha\right)+\exp\left(\beta\right)\right),
\end{equation}
$\forall\left(\alpha,\beta\right)\in\mathbb{R}^{2}$, the absolute
difference on the numerator of (\ref{eq:Diff_of EXPS}) can be upper
bounded as
\begin{flalign}
\left|\mathfrak{N}_{t}-\mathfrak{N}_{t}^{\mathtt{A}}\right| & \le\dfrac{1}{2}\left|{\displaystyle \sum_{i=0}^{t}}\overline{{\bf y}}_{i}^{\boldsymbol{T}}\left(X_{i}\right){\bf C}_{i}^{-1}\left(X_{i}\right)\overline{{\bf y}}_{i}\left(X_{i}\right)-\right.\left.\overline{{\bf y}}_{i}^{\boldsymbol{T}}\left(X_{i}^{\mathtt{A}}\right){\bf C}_{i}^{-1}\left(X_{i}^{\mathtt{A}}\right)\overline{{\bf y}}_{i}\left(X_{i}^{\mathtt{A}}\right)\vphantom{{\displaystyle \sum_{i=0}^{t}}\overline{{\bf y}}_{i}^{\boldsymbol{T}}\left(X_{i}\right){\bf C}_{i}^{-1}\left(X_{i}\right)\overline{{\bf y}}_{i}\left(X_{i}\right)}\right|\left(\mathfrak{N}_{t}+\mathfrak{N}_{t}^{\mathtt{A}}\right)\nonumber \\
 & \le\sum_{i=0}^{t}\left|\overline{{\bf y}}_{i}^{\boldsymbol{T}}\left(X_{i}\right){\bf C}_{i}^{-1}\left(X_{i}\right)\overline{{\bf y}}_{i}\left(X_{i}\right)-\right.\left.\overline{{\bf y}}_{i}^{\boldsymbol{T}}\left(X_{i}^{\mathtt{A}}\right){\bf C}_{i}^{-1}\left(X_{i}^{\mathtt{A}}\right)\overline{{\bf y}}_{i}\left(X_{i}^{\mathtt{A}}\right)\right|.\label{eq:EXPS_LIP}
\end{flalign}
Concentrating on each member of the series above in the last line
of (\ref{eq:EXPS_LIP}) and calling Lemma \ref{Elementary_Lemma_1},
it is true that
\begin{multline}
\left|\overline{{\bf y}}_{i}^{\boldsymbol{T}}\left(X_{i}\right)\left[{\bf C}_{i}^{-1}\left(X_{i}\right)\overline{{\bf y}}_{i}\left(X_{i}\right)\right]-\right.\left.\overline{{\bf y}}_{i}^{\boldsymbol{T}}\left(X_{i}^{\mathtt{A}}\right)\left[{\bf C}_{i}^{-1}\left(X_{i}^{\mathtt{A}}\right)\overline{{\bf y}}_{i}\left(X_{i}^{\mathtt{A}}\right)\right]\right|\\
\le\left\Vert \overline{{\bf y}}_{i}\left(X_{i}\right)\right\Vert _{2}\left\Vert {\bf C}_{i}^{-1}\left(X_{i}\right)\overline{{\bf y}}_{i}\left(X_{i}\right)-{\bf C}_{i}^{-1}\left(X_{i}^{\mathtt{A}}\right)\overline{{\bf y}}_{i}\left(X_{i}^{\mathtt{A}}\right)\right\Vert _{2}+\\
+\left\Vert {\bf C}_{i}^{-1}\left(X_{i}^{\mathtt{A}}\right)\overline{{\bf y}}_{i}\left(X_{i}^{\mathtt{A}}\right)\right\Vert _{2}\left\Vert \overline{{\bf y}}_{i}\left(X_{i}\right)-\overline{{\bf y}}_{i}\left(X_{i}^{\mathtt{A}}\right)\right\Vert _{2}.
\end{multline}
Calling Lemma \ref{Elementary_Lemma_1} again for the term multiplying
the quantity $\left\Vert \overline{{\bf y}}_{i}\left(X_{i}\right)\right\Vert _{2}$
in the RHS of the above expression, we arrive at the inequalities
\begin{flalign}
 & \hspace{-2pt}\hspace{-2pt}\hspace{-2pt}\hspace{-2pt}\hspace{-2pt}\hspace{-2pt}\hspace{-2pt}\hspace{-2pt}\hspace{-2pt}\hspace{-2pt}\hspace{-2pt}\hspace{-2pt}\hspace{-2pt}\hspace{-2pt}\hspace{-2pt}\hspace{-2pt}\hspace{-2pt}\hspace{-2pt}\hspace{-2pt}\hspace{-2pt}\hspace{-2pt}\hspace{-2pt}\hspace{-2pt}\hspace{-2pt}\left|\overline{{\bf y}}_{i}^{\boldsymbol{T}}\left(X_{i}\right){\bf C}_{i}^{-1}\left(X_{i}\right)\overline{{\bf y}}_{i}\left(X_{i}\right)-\right.\left.\overline{{\bf y}}_{i}^{\boldsymbol{T}}\left(X_{i}^{\mathtt{A}}\right){\bf C}_{i}^{-1}\left(X_{i}^{\mathtt{A}}\right)\overline{{\bf y}}_{i}\left(X_{i}^{\mathtt{A}}\right)\right|\nonumber \\
 & \le\left\Vert \overline{{\bf y}}_{i}\left(X_{i}\right)\right\Vert _{2}\left\Vert {\bf C}_{i}^{-1}\left(X_{i}\right)\right\Vert _{2}\left\Vert \overline{{\bf y}}_{i}\left(X_{i}\right)-\overline{{\bf y}}_{i}\left(X_{i}^{\mathtt{A}}\right)\right\Vert _{2}\nonumber \\
 & \quad\quad\quad\quad+\left\Vert \overline{{\bf y}}_{i}\left(X_{i}^{\mathtt{A}}\right)\right\Vert _{2}\left\Vert {\bf C}_{i}^{-1}\left(X_{i}^{\mathtt{A}}\right)\right\Vert _{2}\left\Vert \overline{{\bf y}}_{i}\left(X_{i}\right)-\overline{{\bf y}}_{i}\left(X_{i}^{\mathtt{A}}\right)\right\Vert _{2}\nonumber \\
 & \quad\quad\quad\quad\quad\quad\quad\quad+\left\Vert \overline{{\bf y}}_{i}\left(X_{i}\right)\right\Vert _{2}\left\Vert \overline{{\bf y}}_{i}\left(X_{i}^{\mathtt{A}}\right)\right\Vert _{2}\left\Vert {\bf C}_{i}^{-1}\left(X_{i}\right)-{\bf C}_{i}^{-1}\left(X_{i}^{\mathtt{A}}\right)\right\Vert _{2}\nonumber \\
 & \le\dfrac{\left\Vert \overline{{\bf y}}_{i}\left(X_{i}\right)\right\Vert _{2}}{\lambda_{inf}}\left\Vert \overline{{\bf y}}_{i}\left(X_{i}\right)-\overline{{\bf y}}_{i}\left(X_{i}^{\mathtt{A}}\right)\right\Vert _{2}+\dfrac{\left\Vert \overline{{\bf y}}_{i}\left(X_{i}^{\mathtt{A}}\right)\right\Vert _{2}}{\lambda_{inf}}\left\Vert \overline{{\bf y}}_{i}\left(X_{i}\right)-\overline{{\bf y}}_{i}\left(X_{i}^{\mathtt{A}}\right)\right\Vert _{2}\nonumber \\
 & \quad\quad\quad\quad+\left\Vert \overline{{\bf y}}_{i}\left(X_{i}\right)\right\Vert _{2}\left\Vert \overline{{\bf y}}_{i}\left(X_{i}^{\mathtt{A}}\right)\right\Vert _{2}\left\Vert {\bf C}_{i}^{-1}\left(X_{i}\right)-{\bf C}_{i}^{-1}\left(X_{i}^{\mathtt{A}}\right)\right\Vert _{2},
\end{flalign}
or, equivalently,
\begin{multline}
\left|\overline{{\bf y}}_{i}^{\boldsymbol{T}}\left(X_{i}\right){\bf C}_{i}^{-1}\left(X_{i}\right)\overline{{\bf y}}_{i}\left(X_{i}\right)-\right.\left.\overline{{\bf y}}_{i}^{\boldsymbol{T}}\left(X_{i}^{\mathtt{A}}\right){\bf C}_{i}^{-1}\left(X_{i}^{\mathtt{A}}\right)\overline{{\bf y}}_{i}\left(X_{i}^{\mathtt{A}}\right)\right|\\
\le\dfrac{\left\Vert \overline{{\bf y}}_{i}\left(X_{i}\right)\right\Vert _{2}+\left\Vert \overline{{\bf y}}_{i}\left(X_{i}^{\mathtt{A}}\right)\right\Vert _{2}}{\lambda_{inf}}\left\Vert \overline{{\bf y}}_{i}\left(X_{i}\right)-\overline{{\bf y}}_{i}\left(X_{i}^{\mathtt{A}}\right)\right\Vert _{2}+\\
+\left\Vert \overline{{\bf y}}_{i}\left(X_{i}\right)\right\Vert _{2}\left\Vert \overline{{\bf y}}_{i}\left(X_{i}^{\mathtt{A}}\right)\right\Vert _{2}\left\Vert {\bf C}_{i}^{-1}\left(X_{i}\right)-{\bf C}_{i}^{-1}\left(X_{i}^{\mathtt{A}}\right)\right\Vert _{2}.
\end{multline}
Now, recalling Assumption 2, the definition of $\overline{{\bf y}}_{i}\left(X_{i}\right)$
(resp. for $X_{i}^{\mathtt{A}}$) and invoking Lemma \ref{Elementary_Lemma_3},
it must be true that
\begin{flalign}
 & \hspace{-2pt}\hspace{-2pt}\hspace{-2pt}\hspace{-2pt}\hspace{-2pt}\hspace{-2pt}\hspace{-2pt}\hspace{-2pt}\hspace{-2pt}\hspace{-2pt}\hspace{-2pt}\hspace{-2pt}\left|\overline{{\bf y}}_{i}^{\boldsymbol{T}}\left(X_{i}\right){\bf C}_{i}^{-1}\left(X_{i}\right)\overline{{\bf y}}_{i}\left(X_{i}\right)-\right.\left.\overline{{\bf y}}_{i}^{\boldsymbol{T}}\left(X_{i}^{\mathtt{A}}\right){\bf C}_{i}^{-1}\left(X_{i}^{\mathtt{A}}\right)\overline{{\bf y}}_{i}\left(X_{i}^{\mathtt{A}}\right)\right|\nonumber \\
 & \le\dfrac{\left\Vert {\bf y}_{i}-\boldsymbol{\mu}_{i}\left(X_{i}\right)\right\Vert _{2}+\left\Vert {\bf y}_{i}-\boldsymbol{\mu}_{i}\left(X_{i}^{\mathtt{A}}\right)\right\Vert _{2}}{\lambda_{inf}}\left\Vert \boldsymbol{\mu}_{i}\left(X_{i}\right)-\boldsymbol{\mu}_{i}\left(X_{i}^{\mathtt{A}}\right)\right\Vert _{2}\nonumber \\
 & \quad\quad\quad\quad+\left\Vert {\bf y}_{i}-\boldsymbol{\mu}_{i}\left(X_{i}\right)\right\Vert _{2}\left\Vert {\bf y}_{i}-\boldsymbol{\mu}_{i}\left(X_{i}^{\mathtt{A}}\right)\right\Vert _{2}\left\Vert {\bf C}_{i}^{-1}\left(X_{i}\right)-{\bf C}_{i}^{-1}\left(X_{i}^{\mathtt{A}}\right)\right\Vert _{2}\nonumber \\
 & \le\left(K_{\mu}\dfrac{2\left(\left\Vert {\bf y}_{i}\right\Vert _{2}+\mu_{sup}\right)}{\lambda_{inf}}+K_{INV}\left(\left\Vert {\bf y}_{i}\right\Vert _{2}+\mu_{sup}\right)^{2}\right)\left|X_{i}-X_{i}^{\mathtt{A}}\right|\triangleq\Theta\left({\bf y}_{i}\right)\left|X_{i}-X_{i}^{\mathtt{A}}\right|.\label{eq:Bound_With_Epsilons}
\end{flalign}
Using the above inequality, the RHS of (\ref{eq:Diff_of EXPS}) can
be further bounded from above as
\begin{equation}
\dfrac{\left|\mathfrak{N}_{t}-\mathfrak{N}_{t}^{\mathtt{A}}\right|}{\left|\mathfrak{D}_{t}^{\mathtt{A}}\right|}\le\dfrac{{\displaystyle \sup_{i\in\mathbb{N}_{t}}\Theta\left({\bf y}_{i}\right)}}{\sqrt{\lambda_{inf}^{N\left(t+1\right)}}}\sum_{i=0}^{t}\left|X_{i}-X_{i}^{\mathtt{A}}\right|.
\end{equation}

Therefore, we can bound the RHS of (\ref{eq:CORE_Inequality}) as
\begin{align}
\left|\widehat{\Lambda}_{t}-\widehat{\Lambda}_{t}^{\mathtt{A}}\right| & \le\left(\dfrac{{\displaystyle \sup_{i\in\mathbb{N}_{t}}\Theta\left({\bf y}_{i}\right)}}{\sqrt{\lambda_{inf}^{N\left(t+1\right)}}}+\dfrac{N^{2}K_{DET}K_{\boldsymbol{\Sigma}}}{2\sqrt{\lambda_{inf}^{N\left(t+2\right)}}}\right)\sum_{i=0}^{t}\left|X_{i}-X_{i}^{\mathtt{A}}\right|.\label{eq:CORE_Ineq_2}
\end{align}
Taking conditional expectations on both sides of (\ref{eq:CORE_Ineq_2}),
observing that the quantity $\sup_{i\in\mathbb{N}_{t}}\Theta\left({\bf y}_{i}\right)$
constitutes a $\left\{ \mathscr{Y}_{t}\right\} $-adapted process
and recalling that under the base measure $\widetilde{{\cal P}}$
(see Theorem 2), the processes ${\bf y}_{t}$ and $X_{t}$ (resp.
$X_{t}^{\mathtt{A}}$) are statistically independent, we can write
\begin{align}
\mathbb{E}_{\widetilde{{\cal P}}}\left\{ \left.\left|\widehat{\Lambda}_{t}-\widehat{\Lambda}_{t}^{\mathtt{A}}\right|\right|\mathscr{Y}_{t}\right\}  & \le\left(\dfrac{{\displaystyle \sup_{i\in\mathbb{N}_{t}}\Theta\left({\bf y}_{i}\right)}}{\sqrt{\lambda_{inf}^{N\left(t+1\right)}}}+\dfrac{N^{2}K_{DET}K_{\boldsymbol{\Sigma}}}{2\sqrt{\lambda_{inf}^{N\left(t+2\right)}}}\right)\mathbb{E}_{\widetilde{{\cal P}}}\left\{ \sum_{i=0}^{t}\left|X_{i}-X_{i}^{\mathtt{A}}\right|\right\} ,
\end{align}
$\widetilde{{\cal P}}-a.e.$, and, because ${\cal P}\ll_{\mathscr{H}_{t}}\widetilde{{\cal P}}$,
${\cal P}-a.e.$ as well. From the last inequality, we can readily
observe that in order to be able to talk about any kind of uniform
convergence regarding the RHS, it is vital to ensure that the random
variable $\sup_{i\in\mathbb{N}_{t}}\Theta\left({\bf y}_{i}\right)$
is bounded from above. However, because the support of $\left\Vert {\bf y}_{i}\right\Vert _{2}$
is infinite, it is impossible to bound $\sup_{i\in\mathbb{N}_{t}}\Theta\left({\bf y}_{i}\right)$
in the almost sure sense. Nevertheless, Lemma \ref{Lemma_WHP} immediately
implies that there exists a measurable subset $\widehat{\Omega}_{\tau}\subseteq\Omega$
with
\begin{equation}
\min\left\{ {\cal P}\left(\widehat{\Omega}_{\tau}\right),\widetilde{{\cal P}}\left(\widehat{\Omega}_{\tau}\right)\right\} \ge1-\dfrac{\exp\left(-CN\right)}{\left(\tau+1\right)^{CN-1}}
\end{equation}
 such that, for all $\omega\in\widehat{\Omega}_{\tau}$,
\begin{equation}
\sup_{i\in\mathbb{N}_{\tau}}\left\Vert {\bf y}_{i}\left(\omega\right)\right\Vert _{2}^{2}\equiv\sup_{i\in\mathbb{N}_{\tau}}\left\Vert {\bf y}_{i}\right\Vert _{2}^{2}<\gamma CN\left(1+\log\left(1+\tau\right)\right),
\end{equation}
for some fixed constant $\gamma>1$, for any $C\ge1$ and for any
\textit{fixed} $\tau\in\mathbb{N}$. Choosing $\tau\equiv T<\infty$,
it is true that 
\begin{flalign}
\sup_{i\in\mathbb{N}_{t}}\Theta\left({\bf y}_{i}\right) & \le\sup_{i\in\mathbb{N}_{T}}\Theta\left({\bf y}_{i}\right)\nonumber \\
 & \le\sup_{i\in\mathbb{N}_{T}}\left[K_{\mu}\dfrac{2\left(\left\Vert {\bf y}_{i}\right\Vert _{2}+\mu_{sup}\right)}{\lambda_{inf}}+K_{INV}\left(\left\Vert {\bf y}_{i}\right\Vert _{2}+\mu_{sup}\right)^{2}\right]\nonumber \\
 & <\left(K_{\mu}\dfrac{2\widetilde{\gamma}}{\lambda_{inf}}+K_{INV}\widetilde{\gamma}^{2}\right)CN\left(1+\log\left(1+T\right)\right)\nonumber \\
 & \triangleq K_{o}CN\left(1+\log\left(1+T\right)\right),\quad\forall t\in\mathbb{N}_{T},
\end{flalign}
where $\widetilde{\gamma}\triangleq\sqrt{\gamma}+\mu_{sup}$. Therefore,
it will be true that
\begin{align}
\mathbb{E}_{\widetilde{{\cal P}}}\left\{ \left.\left|\widehat{\Lambda}_{t}-\widehat{\Lambda}_{t}^{\mathtt{A}}\right|\right|\mathscr{Y}_{t}\right\}  & \le\left(\dfrac{K_{o}CN\left(1+\log\left(1+T\right)\right)}{\sqrt{\lambda_{inf}^{N\left(t+1\right)}}}+\dfrac{N^{2}K_{DET}K_{\boldsymbol{\Sigma}}}{2\sqrt{\lambda_{inf}^{N\left(t+2\right)}}}\right)\sum_{i=0}^{t}\mathbb{E}_{\widetilde{{\cal P}}}\left\{ \left|X_{i}-X_{i}^{\mathtt{A}}\right|\right\} ,
\end{align}
for all $t\in\mathbb{N}_{T}$, with probability at least 
\[
1-\dfrac{\exp\left(-CN\right)}{\left(T+1\right)^{CN-1}},
\]
under either ${\cal P}$ or $\widetilde{{\cal P}}$. Further,
\begin{align}
\sum_{i=0}^{t}\mathbb{E}_{\widetilde{{\cal P}}}\left\{ \left|X_{i}-X_{i}^{\mathtt{A}}\right|\right\}  & \le\left(t+1\right)\sup_{\tau\in\mathbb{N}_{t}}\mathbb{E}_{\widetilde{{\cal P}}}\left\{ \left|X_{\tau}-X_{\tau}^{\mathtt{A}}\right|\right\} .
\end{align}
Then, with the same probability of success,
\begin{flalign}
 & \hspace{-2pt}\hspace{-2pt}\hspace{-2pt}\hspace{-2pt}\hspace{-2pt}\hspace{-2pt}\hspace{-2pt}\hspace{-2pt}\hspace{-2pt}\hspace{-2pt}\hspace{-2pt}\hspace{-2pt}\mathbb{E}_{\widetilde{{\cal P}}}\left\{ \left.\left|\widehat{\Lambda}_{t}-\widehat{\Lambda}_{t}^{\mathtt{A}}\right|\right|\mathscr{Y}_{t}\right\} \nonumber \\
 & \le\left(\dfrac{K_{o}CN\left(1+\log\left(1+T\right)\right)\left(t+1\right)}{\sqrt{\lambda_{inf}^{N\left(t+1\right)}}}+\right.\left.\dfrac{K_{DET}K_{\boldsymbol{\Sigma}}N^{2}\left(t+1\right)}{2\sqrt{\lambda_{inf}^{N\left(t+2\right)}}}\right)\sup_{\tau\in\mathbb{N}_{t}}\mathbb{E}_{\widetilde{{\cal P}}}\left\{ \left|X_{\tau}-X_{\tau}^{\mathtt{A}}\right|\right\} \nonumber \\
 & \le\left(\dfrac{K_{o}CN\left(1+\log\left(1+T\right)\right)\left(T+1\right)}{\lambda_{inf}^{N/2}}+\right.\left.\dfrac{K_{DET}K_{\boldsymbol{\Sigma}}N^{2}\left(T+1\right)}{2\lambda_{inf}^{N}}\right)\sup_{\tau\in\mathbb{N}_{t}}\mathbb{E}_{\widetilde{{\cal P}}}\left\{ \left|X_{\tau}-X_{\tau}^{\mathtt{A}}\right|\right\} \nonumber \\
 & \triangleq K_{G}\left(T\right)\sup_{\tau\in\mathbb{N}_{t}}\mathbb{E}_{\widetilde{{\cal P}}}\left\{ \left|X_{\tau}-X_{\tau}^{\mathtt{A}}\right|\right\} ,\quad\forall t\in\mathbb{N}_{T},
\end{flalign}
where $K_{G}\left(T\right)\equiv{\cal O}\left(T\log\left(T\right)\right)$.
\textit{Alternatively}, upper bounding the functions comprised by
the quantities $t,N,\lambda_{inf}$ in the second and third lines
of the expressions above as (note that, obviously, $t+1\ge1,\forall t\in\mathbb{R}_{+}$)
\begin{flalign}
\dfrac{N\left(t+1\right)}{\sqrt{\lambda_{inf}^{N\left(t+1\right)}}} & \le\max_{t\in\mathbb{R}_{+}}\dfrac{N\left(t+1\right)^{2}}{\sqrt{\lambda_{inf}^{N\left(t+1\right)}}}\quad\text{and}\\
\dfrac{N^{2}\left(t+1\right)}{\sqrt{\lambda_{inf}^{N\left(t+2\right)}}} & \le\max_{t\in\mathbb{R}_{+}}\dfrac{N^{2}\left(t+1\right)^{2}}{\sqrt{\lambda_{inf}^{N\left(t+2\right)}}},
\end{flalign}
respectively, we can also define
\begin{align}
K_{G}\left(T\right) & \triangleq\dfrac{16K_{o}C\left(1+\log\left(1+T\right)\right)\lambda_{inf}^{-2/\log\left(\lambda_{inf}\right)}}{N\left(\log\left(\lambda_{inf}\right)\right)^{2}}+\dfrac{8K_{DET}K_{\boldsymbol{\Sigma}}\lambda_{inf}^{-N}\lambda_{inf}^{-2/\log\left(\lambda_{inf}\right)}}{\left(\log\left(\lambda_{inf}\right)\right)^{2}},
\end{align}
where, in this case, $K_{G}\left(T\right)\equiv{\cal O}\left(\log\left(T\right)\right)$.
Note, however, that although its dependence on $T$ is logarithmic,
$K_{G}\left(T\right)$ may still be large due to the inability to
compensate for the size of $K_{o}$. In any case, for all $\omega\in\widehat{\Omega}_{T}$,
\begin{align}
\mathbb{E}_{\widetilde{{\cal P}}}\left\{ \left.\left|\widehat{\Lambda}_{t}-\widehat{\Lambda}_{t}^{\mathtt{A}}\right|\right|\mathscr{Y}_{t}\right\} \left(\omega\right) & \le K_{G}\left(T\right)\sup_{\tau\in\mathbb{N}_{t}}\mathbb{E}_{\widetilde{{\cal P}}}\left\{ \left|X_{\tau}-X_{\tau}^{\mathtt{A}}\right|\right\} ,\quad\forall t\in\mathbb{N}_{T},
\end{align}
Therefore, we get
\begin{align}
\sup_{\omega\in\widehat{\Omega}_{T}}\mathbb{E}_{\widetilde{{\cal P}}}\left\{ \left.\left|\widehat{\Lambda}_{t}-\widehat{\Lambda}_{t}^{\mathtt{A}}\right|\right|\mathscr{Y}_{t}\right\} \left(\omega\right) & \le K_{G}\left(T\right)\sup_{\tau\in\mathbb{N}_{t}}\mathbb{E}_{\widetilde{{\cal P}}}\left\{ \left|X_{\tau}-X_{\tau}^{\mathtt{A}}\right|\right\} ,\quad\forall t\in\mathbb{N}_{T}
\end{align}
and further taking the supremum over $t\in\mathbb{N}_{T}$ on both
sides, it must be true that
\begin{flalign}
\sup_{t\in\mathbb{N}_{T}}\sup_{\omega\in\widehat{\Omega}_{T}}\mathbb{E}_{\widetilde{{\cal P}}}\left\{ \left.\left|\widehat{\Lambda}_{t}-\widehat{\Lambda}_{t}^{\mathtt{A}}\right|\right|\mathscr{Y}_{t}\right\} \left(\omega\right) & \le K_{G}\left(T\right)\sup_{t\in\mathbb{N}_{T}}\sup_{\tau\in\mathbb{N}_{t}}\mathbb{E}_{\widetilde{{\cal P}}}\left\{ \left|X_{\tau}-X_{\tau}^{\mathtt{A}}\right|\right\} \nonumber \\
 & \equiv K_{G}\left(T\right)\sup_{t\in\mathbb{N}_{T}}\mathbb{E}_{\widetilde{{\cal P}}}\left\{ \left|X_{t}-X_{t}^{\mathtt{A}}\right|\right\} .\label{eq:CORE_Final}
\end{flalign}
Finally, if either
\begin{equation}
{\cal P}_{\left.X_{t}^{\mathtt{A}}\right|X_{t}}^{\mathtt{A}}\left(\left.\cdot\right|X_{t}\right)\stackrel[\mathtt{A}\rightarrow\infty]{{\cal W}}{\longrightarrow}\delta_{X_{t}}\left(\cdot\right)\equiv\mathds{1}_{\left(\cdot\right)}\left(X_{t}\right),\quad\forall t\in\mathbb{N}_{T},
\end{equation}
or 
\begin{equation}
X_{t}^{\mathtt{A}}\stackrel[\mathtt{A}\rightarrow\infty]{{\cal P}}{\longrightarrow}X_{t},\quad\forall t\in\mathbb{N}_{T}
\end{equation}
and given that since the members of $\left\{ X_{t}^{\mathtt{A}}\right\} _{\mathtt{A}\in\mathbb{N}}$
are almost surely bounded in ${\cal Z}$, the aforementioned sequence
is also uniformly integrable for all $t\in\mathbb{N}$, it must be
true that (see Lemma \ref{L1WEAK_Lemma})
\begin{equation}
\mathbb{E}_{\widetilde{{\cal P}}}\left\{ \left|X_{t}-X_{t}^{\mathtt{A}}\right|\right\} \underset{\mathtt{A}\rightarrow\infty}{\longrightarrow}0,\quad\forall t\in\mathbb{N}_{T}.
\end{equation}
Then, Lemma \ref{L1WEAK_LEMMA} implies that 
\begin{equation}
\sup_{t\in\mathbb{N}_{T}}\mathbb{E}_{\widetilde{{\cal P}}}\left\{ \left|X_{t}-X_{t}^{\mathtt{A}}\right|\right\} \underset{\mathtt{A}\rightarrow\infty}{\longrightarrow}0,
\end{equation}
which in turn implies the existence of the limit on the LHS of (\ref{eq:CORE_Final}).
QED.
\end{proof}

\subsection{\noindent Finishing the Proof of Theorem \ref{CONVERGENCE_THEOREM}}

Considering the absolute difference of the RHSs of (\ref{eq:CoM_Approx})
and (\ref{eq:CoM_Original}), it is true that (see Lemma \ref{Elementary_Lemma_1})
\begin{align}
\left|\dfrac{\mathbb{E}_{\widetilde{{\cal P}}}\left\{ \left.X_{t}\Lambda_{t}\right|\mathscr{Y}_{t}\right\} }{\mathbb{E}_{\widetilde{{\cal P}}}\left\{ \left.\Lambda_{t}\right|\mathscr{Y}_{t}\right\} }-\dfrac{\mathbb{E}_{\widetilde{{\cal P}}}\left\{ \left.X_{t}^{\mathtt{A}}\Lambda_{t}^{\mathtt{A}}\right|\mathscr{Y}_{t}\right\} }{\mathbb{E}_{\widetilde{{\cal P}}}\left\{ \left.\Lambda_{t}^{\mathtt{A}}\right|\mathscr{Y}_{t}\right\} }\right| & \equiv\left|\dfrac{\mathbb{E}_{\widetilde{{\cal P}}}\left\{ \left.X_{t}\widehat{\Lambda}_{t}\right|\mathscr{Y}_{t}\right\} }{\mathbb{E}_{\widetilde{{\cal P}}}\left\{ \left.\widehat{\Lambda}_{t}\right|\mathscr{Y}_{t}\right\} }-\dfrac{\mathbb{E}_{\widetilde{{\cal P}}}\left\{ \left.X_{t}^{\mathtt{A}}\widehat{\Lambda}_{t}^{\mathtt{A}}\right|\mathscr{Y}_{t}\right\} }{\mathbb{E}_{\widetilde{{\cal P}}}\left\{ \left.\widehat{\Lambda}_{t}^{\mathtt{A}}\right|\mathscr{Y}_{t}\right\} }\right|,
\end{align}
due to the fact that the increasing stochastic process
\begin{equation}
{\displaystyle \prod_{i\in\mathbb{N}_{t}}}\exp\left(\dfrac{1}{2}\left\Vert {\bf y}_{i}\right\Vert _{2}^{2}\right)\equiv\exp\left(\dfrac{1}{2}{\displaystyle \sum_{i\in\mathbb{N}_{t}}}\left\Vert {\bf y}_{i}\right\Vert _{2}^{2}\right)
\end{equation}
is $\left\{ \mathscr{Y}_{t}\right\} $-adapted. Then, we can write
\begin{flalign}
 & \hspace{-2pt}\hspace{-2pt}\hspace{-2pt}\hspace{-2pt}\hspace{-2pt}\hspace{-2pt}\hspace{-2pt}\hspace{-2pt}\hspace{-2pt}\hspace{-2pt}\hspace{-2pt}\hspace{-2pt}\left|\dfrac{\mathbb{E}_{\widetilde{{\cal P}}}\left\{ \left.X_{t}\widehat{\Lambda}_{t}\right|\mathscr{Y}_{t}\right\} }{\mathbb{E}_{\widetilde{{\cal P}}}\left\{ \left.\widehat{\Lambda}_{t}\right|\mathscr{Y}_{t}\right\} }-\dfrac{\mathbb{E}_{\widetilde{{\cal P}}}\left\{ \left.X_{t}^{\mathtt{A}}\widehat{\Lambda}_{t}^{\mathtt{A}}\right|\mathscr{Y}_{t}\right\} }{\mathbb{E}_{\widetilde{{\cal P}}}\left\{ \left.\widehat{\Lambda}_{t}^{\mathtt{A}}\right|\mathscr{Y}_{t}\right\} }\right|\nonumber \\
 & \equiv\dfrac{\left|\mathbb{E}_{\widetilde{{\cal P}}}\left\{ \left.X_{t}\widehat{\Lambda}_{t}\right|\mathscr{Y}_{t}\right\} \mathbb{E}_{\widetilde{{\cal P}}}\left\{ \left.\widehat{\Lambda}_{t}^{\mathtt{A}}\right|\mathscr{Y}_{t}\right\} -\right.\left.\mathbb{E}_{\widetilde{{\cal P}}}\left\{ \left.X_{t}^{\mathtt{A}}\widehat{\Lambda}_{t}^{\mathtt{A}}\right|\mathscr{Y}_{t}\right\} \mathbb{E}_{\widetilde{{\cal P}}}\left\{ \left.\widehat{\Lambda}_{t}\right|\mathscr{Y}_{t}\right\} \right|}{\mathbb{E}_{\widetilde{{\cal P}}}\left\{ \left.\widehat{\Lambda}_{t}\right|\mathscr{Y}_{t}\right\} \mathbb{E}_{\widetilde{{\cal P}}}\left\{ \left.\widehat{\Lambda}_{t}^{\mathtt{A}}\right|\mathscr{Y}_{t}\right\} }\nonumber \\
 & \le\dfrac{\left|\mathbb{E}_{{\cal P}}\left\{ \left.X_{t}\right|\mathscr{Y}_{t}\right\} \right|\left|\mathbb{E}_{\widetilde{{\cal P}}}\left\{ \left.\widehat{\Lambda}_{t}-\widehat{\Lambda}_{t}^{\mathtt{A}}\right|\mathscr{Y}_{t}\right\} \right|}{\mathbb{E}_{\widetilde{{\cal P}}}\left\{ \left.\widehat{\Lambda}_{t}^{\mathtt{A}}\right|\mathscr{Y}_{t}\right\} }+\dfrac{\left|\mathbb{E}_{\widetilde{{\cal P}}}\left\{ \left.X_{t}\widehat{\Lambda}_{t}-X_{t}^{\mathtt{A}}\widehat{\Lambda}_{t}^{\mathtt{A}}\right|\mathscr{Y}_{t}\right\} \right|}{\mathbb{E}_{\widetilde{{\cal P}}}\left\{ \left.\widehat{\Lambda}_{t}^{\mathtt{A}}\right|\mathscr{Y}_{t}\right\} },\quad\widetilde{{\cal P}},{\cal P}-a.e..\label{eq:Expected_Differences}
\end{flalign}
Let us first focus on the difference on the numerator of the second
ratio of the RHS of (\ref{eq:Expected_Differences}). Recalling that
$\delta\equiv\max\left\{ \left|a\right|,\left|b\right|\right\} $,
we can then write
\begin{flalign}
\left|\mathbb{E}_{\widetilde{{\cal P}}}\left\{ \left.X_{t}\widehat{\Lambda}_{t}-X_{t}^{\mathtt{A}}\widehat{\Lambda}_{t}^{\mathtt{A}}\right|\mathscr{Y}_{t}\right\} \right| & \le\mathbb{E}_{\widetilde{{\cal P}}}\left\{ \left.\left|X_{t}\widehat{\Lambda}_{t}-X_{t}^{\mathtt{A}}\widehat{\Lambda}_{t}^{\mathtt{A}}\right|\right|\mathscr{Y}_{t}\right\} \nonumber \\
 & \le\mathbb{E}_{\widetilde{{\cal P}}}\left\{ \left.\left|X_{t}\right|\left|\widehat{\Lambda}_{t}-\widehat{\Lambda}_{t}^{\mathtt{A}}\right|+\left|\widehat{\Lambda}_{t}^{\mathtt{A}}\right|\left|X_{t}-X_{t}^{\mathtt{A}}\right|\right|\mathscr{Y}_{t}\right\} \nonumber \\
 & \le\delta\mathbb{E}_{\widetilde{{\cal P}}}\left\{ \left.\left|\widehat{\Lambda}_{t}-\widehat{\Lambda}_{t}^{\mathtt{A}}\right|\right|\mathscr{Y}_{t}\right\} +\mathbb{E}_{\widetilde{{\cal P}}}\left\{ \left.\left|X_{t}-X_{t}^{\mathtt{A}}\right|\right|\mathscr{Y}_{t}\right\} \nonumber \\
 & \equiv\delta\mathbb{E}_{\widetilde{{\cal P}}}\left\{ \left.\left|\widehat{\Lambda}_{t}-\widehat{\Lambda}_{t}^{\mathtt{A}}\right|\right|\mathscr{Y}_{t}\right\} +\mathbb{E}_{\widetilde{{\cal P}}}\left\{ \left|X_{t}-X_{t}^{\mathtt{A}}\right|\right\} ,
\end{flalign}
On the other hand, for the denominator for (\ref{eq:Expected_Differences}),
it is true that
\begin{flalign}
\mathbb{E}_{\widetilde{{\cal P}}}\left\{ \left.\widehat{\Lambda}_{t}^{\mathtt{A}}\right|\mathscr{Y}_{t}\right\}  & \equiv\mathbb{E}_{\widetilde{{\cal P}}}\left\{ \left.\dfrac{\exp\left(-\dfrac{1}{2}{\displaystyle \sum_{i\in\mathbb{N}_{t}}}\overline{{\bf y}}_{i}^{\boldsymbol{T}}\left(X_{i}^{\mathtt{A}}\right){\bf C}_{i}^{-1}\left(X_{i}^{\mathtt{A}}\right)\overline{{\bf y}}_{i}\left(X_{i}^{\mathtt{A}}\right)\right)}{{\displaystyle \prod_{i\in\mathbb{N}_{t}}}\sqrt{\det\left({\bf C}_{i}\left(X_{i}^{\mathtt{A}}\right)\right)}}\right|\mathscr{Y}_{t}\right\} \nonumber \\
 & \ge\dfrac{\mathbb{E}_{\widetilde{{\cal P}}}\left\{ \left.\exp\left(-\dfrac{1}{2}{\displaystyle \sum_{i\in\mathbb{N}_{t}}}\overline{{\bf y}}_{i}^{\boldsymbol{T}}\left(X_{i}^{\mathtt{A}}\right){\bf C}_{i}^{-1}\left(X_{i}^{\mathtt{A}}\right)\overline{{\bf y}}_{i}\left(X_{i}^{\mathtt{A}}\right)\right)\right|\mathscr{Y}_{t}\right\} }{\sqrt{\lambda_{sup}^{N\left(t+1\right)}}}\nonumber \\
 & \ge\dfrac{\mathbb{E}_{\widetilde{{\cal P}}}\left\{ \left.\exp\left(-\dfrac{1}{2\lambda_{inf}}{\displaystyle \sum_{i\in\mathbb{N}_{t}}\left(\left\Vert {\bf y}_{i}\right\Vert _{2}+\mu_{sup}\right)^{2}}\right)\right|\mathscr{Y}_{t}\right\} }{\sqrt{\lambda_{sup}^{N\left(t+1\right)}}}\nonumber \\
 & \equiv\dfrac{\exp\left(-\dfrac{1}{2\lambda_{inf}}{\displaystyle \sum_{i\in\mathbb{N}_{t}}}\left(\left\Vert {\bf y}_{i}\right\Vert _{2}+\mu_{sup}\right)^{2}\right)}{\sqrt{\lambda_{sup}^{N\left(t+1\right)}}},\quad\widetilde{{\cal P}},{\cal P}-a.e.,
\end{flalign}
since the process ${\displaystyle \sum_{i\in\mathbb{N}_{t}}\left(\left\Vert {\bf y}_{i}\right\Vert _{2}+\mu_{sup}\right)}^{2}$
is $\left\{ \mathscr{Y}_{t}\right\} $-adapted. Now, from Lemma \ref{Lemma_WHP},
we know that 
\begin{equation}
\sup_{i\in\mathbb{N}_{t}}\left\Vert {\bf y}_{i}\right\Vert _{2}^{2}\le\sup_{i\in\mathbb{N}_{T}}\left\Vert {\bf y}_{i}\right\Vert _{2}^{2}<\gamma CN\left(1+\log\left(T+1\right)\right),\quad\forall t\in\mathbb{N}_{T},
\end{equation}
where the last inequality holds with probability at least $1-\left(T+1\right)^{1-CN}\exp\left(-CN\right)$,
under both base measures ${\cal P}$ and $\widetilde{{\cal P}}$,
for any finite constant $C\ge1$. Therefore, it can be trivially shown
that
\begin{align}
\mathbb{E}_{\widetilde{{\cal P}}}\left\{ \left.\widehat{\Lambda}_{t}^{\mathtt{A}}\right|\mathscr{Y}_{t}\right\}  & \ge\dfrac{\exp\left(-\dfrac{\left(\sqrt{\gamma CN\left(1+\log\left(T+1\right)\right)}+\mu_{sup}\right)^{2}\left(T+1\right)}{2\lambda_{inf}}\right)}{\sqrt{\lambda_{sup}^{N\left(t+1\right)}}}>0,\;\forall t\in\mathbb{N}_{T},
\end{align}
implying that
\begin{equation}
\inf_{t\in\mathbb{N}_{T}}\inf_{\omega\in\widehat{\Omega}_{T}}{\displaystyle \inf_{\mathtt{A}\in\mathbb{N}}\mathbb{E}_{\widetilde{{\cal P}}}\left\{ \left.\widehat{\Lambda}_{t}^{\mathtt{A}}\right|\mathscr{Y}_{t}\right\} \left(\omega\right)}>0,
\end{equation}
where $\widehat{\Omega}_{T}$ coincides with the event
\[
\left\{ \vphantom{\omega\in\Omega\left|\sup_{i\in\mathbb{N}_{t}}\left\Vert {\bf y}_{i}\right\Vert _{2}^{2}\right.}\omega\in\Omega\left|\sup_{i\in\mathbb{N}_{t}}\left\Vert {\bf y}_{i}\right\Vert _{2}^{2}<\right.\right.\left.\gamma CN\left(1+\log\left(T+1\right)\right),\forall t\in\mathbb{N}_{T}\vphantom{\omega\in\Omega\left|\sup_{i\in\mathbb{N}_{t}}\left\Vert {\bf y}_{i}\right\Vert _{2}^{2}\right.}\right\} 
\]
with ${\cal P}$,$\widetilde{{\cal P}}$-measure at least $1-\left(T+1\right)^{1-CN}\exp\left(-CN\right)$.
Of course, the existence of $\widehat{\Omega}_{T}$ follows from Lemma
\ref{Lemma_WHP}. Putting it altogether, (\ref{eq:Expected_Differences})
becomes (recall that the base measures ${\cal P}$ and $\widetilde{{\cal P}}$
are equivalent)
\begin{align}
\left|\mathbb{E}_{{\cal P}}\left\{ \left.X_{t}\right|\mathscr{Y}_{t}\right\} -{\cal E}^{\mathtt{A}}\left(\left.X_{t}\right|\mathscr{Y}_{t}\right)\right| & \le\dfrac{2\delta\mathbb{E}_{\widetilde{{\cal P}}}\left\{ \left.\left|\widehat{\Lambda}_{t}-\widehat{\Lambda}_{t}^{\mathtt{A}}\right|\right|\mathscr{Y}_{t}\right\} +\mathbb{E}_{\widetilde{{\cal P}}}\left\{ \left|X_{t}-X_{t}^{\mathtt{A}}\right|\right\} }{{\displaystyle \inf_{\mathtt{A}\in\mathbb{N}}\mathbb{E}_{\widetilde{{\cal P}}}\left\{ \left.\widehat{\Lambda}_{t}^{\mathtt{A}}\right|\mathscr{Y}_{t}\right\} }},\quad{\cal P},\widetilde{{\cal P}}-a.e.,
\end{align}
where $\delta\equiv\max\left\{ \left|a\right|,\left|b\right|\right\} $.
Taking the supremum both with respect to $\omega\in\widehat{\Omega}_{T}$
and $t\in\mathbb{N}_{T}$ on both sides, we get
\begin{flalign}
 & \hspace{-2pt}\hspace{-2pt}\hspace{-2pt}\hspace{-2pt}\hspace{-2pt}\hspace{-2pt}\hspace{-2pt}\hspace{-2pt}\hspace{-2pt}\hspace{-2pt}\hspace{-2pt}\hspace{-2pt}\hspace{-2pt}\hspace{-2pt}\hspace{-2pt}\hspace{-2pt}\hspace{-2pt}\hspace{-2pt}\hspace{-2pt}\hspace{-2pt}\hspace{-2pt}\hspace{-2pt}\hspace{-2pt}\hspace{-2pt}\hspace{-2pt}\hspace{-2pt}\hspace{-2pt}\hspace{-2pt}\hspace{-2pt}\hspace{-2pt}{\displaystyle \sup_{t\in\mathbb{N}_{T}}}{\displaystyle \sup_{\omega\in\widehat{\Omega}_{T}}}\left|\mathbb{E}_{{\cal P}}\left\{ \left.X_{t}\right|\mathscr{Y}_{t}\right\} \left(\omega\right)-{\cal E}^{\mathtt{A}}\left(\left.X_{t}\right|\mathscr{Y}_{t}\right)\left(\omega\right)\right|\nonumber \\
 & \le\dfrac{{\displaystyle \sup_{t\in\mathbb{N}_{T}}}\mathbb{E}_{\widetilde{{\cal P}}}\left\{ \left|X_{t}-X_{t}^{\mathtt{A}}\right|\right\} +2\delta{\displaystyle \sup_{t\in\mathbb{N}_{T}}}{\displaystyle \sup_{\omega\in\widehat{\Omega}_{T}}}\mathbb{E}_{\widetilde{{\cal P}}}\left\{ \left.\left|\widehat{\Lambda}_{t}-\widehat{\Lambda}_{t}^{\mathtt{A}}\right|\right|\mathscr{Y}_{t}\right\} \left(\omega\right)}{{\displaystyle \inf_{t\in\mathbb{N}_{T}}}{\displaystyle \inf_{\omega\in\widehat{\Omega}_{T}}}{\displaystyle \inf_{\mathtt{A}\in\mathbb{N}}\mathbb{E}_{\widetilde{{\cal P}}}\left\{ \left.\widehat{\Lambda}_{t}^{\mathtt{A}}\right|\mathscr{Y}_{t}\right\} \left(\omega\right)}},\label{eq:Converg_Thm_Proof_1}
\end{flalign}
with 
\begin{align}
\min\left\{ {\cal P}\left(\widehat{\Omega}_{T}\right),\widetilde{{\cal P}}\left(\widehat{\Omega}_{T}\right)\right\}  & \ge1-\dfrac{\exp\left(-CN\right)}{\left(T+1\right)^{CN-1}}.
\end{align}
Finally, calling Lemma \ref{RD_Derivatives_Lemma} and Lemma \ref{L1WEAK_LEMMA},
and since the denominators of the fractions appearing in (\ref{eq:Converg_Thm_Proof_1})
are nonzero, its RHS tends to zero as $\mathtt{A}\rightarrow\infty$,
under the respective hypotheses. Consequently, the LHS will also converge,
therefore completing the proof of the Theorem \ref{CONVERGENCE_THEOREM}.\hfill{}\ensuremath{\blacksquare}

\section{Conclusion}

In this paper, we have provided sufficient conditions for convergence
of approximate, asymptotically optimal nonlinear filtering operators,
for a general class of hidden stochastic processes, observed in a
conditionally Gaussian noisy environment. In particular, employing
a common change of measure argument, we have shown that using the
same measurements, but replacing the ``true'' state by an approximation
process, which converges to the former either in probability or in
the ${\cal C}$-weak sense, one can define an approximate filtering
operator, which converges to the optimal filter compactly in time
and uniformly in an event occurring with probability nearly $1$,
at the same time constituting a purely quantitative justification
of Egoroff's theorem for the problem of interest. The results presented
in this paper essentially provide a framework for analyzing the convergence
properties of various classes of approximate nonlinear filters (either
recursive or nonrecursive), such as existing grid based approaches,
which are known to perform well in various applications.

\section*{Appendix: Proof of Theorem \ref{thm:(Conditional-Bayes'-Theorem)}}

The proof is astonishingly simple. Let the hypotheses of the statement
of Theorem 1 hold true. To avoid useless notational congestion, let
us also make the identifications
\begin{flalign}
{\cal X}_{t} & \triangleq\left\{ X_{i}\right\} _{i\in\mathbb{N}_{t}}\quad\text{and}\quad{\cal Y}_{t}\triangleq\left\{ Y_{i}\right\} _{i\in\mathbb{N}_{t}}.\label{eq:X_Cal}
\end{flalign}
Now, by definition of the conditional expectation operator and since
we have assumed the existence of densities, it is true that
\begin{flalign}
\widehat{X}_{t} & \equiv\mathbb{E}_{{\cal P}}\left\{ \left.X_{t}\right|Y_{0},Y_{1},\ldots,Y_{t}\right\} \nonumber \\
 & =\int x_{t}f_{\left.X_{t}\right|{\cal Y}_{t}}\left(\left.x_{t}\right|{\cal Y}_{t}\right)\mathrm{d}x_{t}\nonumber \\
 & =\dfrac{{\displaystyle \int}x_{t}f_{\left(X_{t},{\cal Y}_{t}\right)}\left(x_{t},{\cal Y}_{t}\right)\mathrm{d}x_{t}}{f_{{\cal Y}_{t}}\left({\cal Y}_{t}\right)}\nonumber \\
 & \equiv\dfrac{\mathop{\mathlarger{\mathlarger{\int}}}x_{t}f_{t}\left(x_{0},x_{1},\ldots x_{t},{\cal Y}_{t}\right){\displaystyle \prod_{i=0}^{t}}\mathrm{d}x_{i}}{\mathop{\mathlarger{\mathlarger{\int}}}f_{t}\left(x_{0},x_{1},\ldots x_{t},{\cal Y}_{t}\right){\displaystyle \prod_{i=0}^{t}}\mathrm{d}x_{i}}\nonumber \\
 & \equiv\dfrac{\mathop{\mathlarger{\mathlarger{\int}}}x_{t}\lambda_{t}\widetilde{f}_{t}\left(\left\{ x_{i}\right\} _{i\in\mathbb{N}_{t}},{\cal Y}_{t}\right){\displaystyle \prod_{i=0}^{t}}\mathrm{d}x_{i}}{\mathop{\mathlarger{\mathlarger{\int}}}\lambda_{t}\widetilde{f}_{t}\left(\left\{ x_{i}\right\} _{i\in\mathbb{N}_{t}},{\cal Y}_{t}\right){\displaystyle \prod_{i=0}^{t}}\mathrm{d}x_{i}},\label{eq:Pre_Bayes}
\end{flalign}
where
\begin{flalign}
\lambda_{t} & \triangleq\dfrac{f_{t}\left(x_{0},x_{1},\ldots x_{t},Y_{0}\left(\omega\right),Y_{1}\left(\omega\right),\ldots,Y_{t}\left(\omega\right)\right)}{\widetilde{f}_{t}\left(x_{0},x_{1},\ldots x_{t},Y_{0}\left(\omega\right),Y_{1}\left(\omega\right),\ldots,Y_{t}\left(\omega\right)\right)}\nonumber \\
 & \equiv\lambda_{t}\left(\left\{ x_{i}\right\} _{i\in\mathbb{N}_{t}},{\cal Y}_{t}\left(\omega\right)\right)\in\mathbb{R}_{+},\quad\forall\omega\in\Omega,
\end{flalign}
constitutes a ``half ordinary function - half random variable''
likelihood ratio and where the condition (\ref{eq:ABSOLUTE_Densities})
ensures its boundedness. Of course, although the likelihood ratio
can be indeterminate when both densities are zero, the respective
points do not contribute in the computation of the relevant integrals
presented above, because these belong to measurable sets corresponding
to events of measure zero. From (\ref{eq:Pre_Bayes}) and by definition
of the conditional density of $\left\{ X_{i}\right\} _{i\in\mathbb{N}_{t}}$
given $\left\{ Y_{i}\right\} _{i\in\mathbb{N}_{t}}$, we immediately
get 
\begin{flalign}
\widehat{X}_{t} & \equiv\mathbb{E}_{{\cal P}}\left\{ \left.X_{t}\right|\mathscr{Y}_{t}\right\} \nonumber \\
 & \equiv\dfrac{\mathop{\mathlarger{\mathlarger{\int}}}x_{t}\lambda_{t}\widetilde{f}_{\left.{\cal X}_{t}\right|{\cal Y}_{t}}\left(\left.\left\{ x_{i}\right\} _{i\in\mathbb{N}_{t}}\right|{\cal Y}_{t}\right){\displaystyle \prod_{i=0}^{t}}\mathrm{d}x_{i}}{\mathop{\mathlarger{\mathlarger{\int}}}\lambda_{t}\widetilde{f}_{\left.{\cal X}_{t}\right|{\cal Y}_{t}}\left(\left.\left\{ x_{i}\right\} _{i\in\mathbb{N}_{t}}\right|{\cal Y}_{t}\right){\displaystyle \prod_{i=0}^{t}}\mathrm{d}x_{i}}\nonumber \\
 & =\dfrac{\mathbb{E}_{\widetilde{{\cal P}}}\left\{ \left.X_{t}\Lambda_{t}\right|Y_{0},Y_{1},\ldots,Y_{t}\right\} }{\mathbb{E}_{\widetilde{{\cal P}}}\left\{ \left.\Lambda_{t}\right|Y_{0},Y_{1},\ldots,Y_{t}\right\} }\nonumber \\
 & \equiv\dfrac{\mathbb{E}_{\widetilde{{\cal P}}}\left\{ \left.X_{t}\Lambda_{t}\right|\mathscr{Y}_{t}\right\} }{\mathbb{E}_{\widetilde{{\cal P}}}\left\{ \left.\Lambda_{t}\right|\mathscr{Y}_{t}\right\} },
\end{flalign}
which constitutes what we were initially set to show.\hfill{}\ensuremath{\blacksquare}

\bibliographystyle{IEEEtran}
\bibliography{IEEEabrv}

\end{document}